\documentclass[11pt]{amsart}
\usepackage{graphicx}
\usepackage{amsmath}
\usepackage{amsfonts}
\usepackage{amssymb}
\usepackage{xcolor}
\newtheorem{theorem}{Theorem}[section]

\newtheorem{conjecture}[theorem]{Conjecture}
\newtheorem{corollary}[theorem]{Corollary}

\newtheorem{definition}[theorem]{Definition}

\newtheorem{lemma}[theorem]{Lemma}
\newtheorem{notation}[theorem]{Notation}

\newtheorem{proposition}[theorem]{Proposition}
\newtheorem{remark}[theorem]{Remark}

\newtheorem{thm}{Theorem}[section]

\newtheorem{cor}[thm]{Corollary}

\newtheorem{lem}[thm]{Lemma}

\newtheorem{prop}[thm]{Proposition}

\theoremstyle{definition}
\newtheorem{defn}[thm]{Definition}

\def\C{\mbox{${\mathbb C}$}}
\def\D{\mbox{${\mathbb D}$}}
\def\T{\mbox{${\mathbb T}$}}

\def\K{\mbox{${\mathbb K}$}}

\def\Mob{M\"{o}b }

\def\reps{{representations~}}%
\def\pr{{projective~}}%
\def\con{{contraction~}}%
\def\cf{{characteristic function}}%

\begin{document}

\title[A product formula]{A product formula for homogeneous\\ characteristic functions}

\author[B. Bagchi]{Bhaskar Bagchi}
\address[B. Bagchi]{Statistics and Mathematics Unit, Indian Statistical Institute, Bangalore 560 059} 
\curraddr[B. Bagchi]{\#1363, 10th Cross Road, Kengeri Satellite Town, Bangalore 560060}
\email[B. Bagchi]{bbagchi@isibang.ac.in \mbox{\emph{and}} bhaskarbagchi53@gmail.com} 
\author[S. Hazra]{Somnath Hazra}
\author[G. Misra]{Gadadhar Misra}
\address[S. Hazra and G. Misra]{Department of Mathematics, Indian Institute of Science,  Bangalore 560012}
\curraddr[S.  Hazra]{Department of Mathematics and Statistics, Indian Institute of Science Education \phantom{aaa}and Research Kolkata, Mohanpur, Nadia - 741 246}
\email[G. Misra]{gm@iisc.ac.in}
\email[S. Hazra]{somnath.hazra.2008@gmail.com}

\dedicatory{This paper is dedicated to the memory of Ronald G. Douglas}
\keywords{Pure contractions, Sz.-Nagy--Foias characteristic function, Defect spaces, M\"{o}bius group, Homogeneous operators, Associator,  Projective representations, equivariant model theory}

\subjclass[2010]{Primary 47A20, 47A45; Secondary 47B32, 20C25}

\thanks{The work of G. Misra was supported, in part, through the J C Bose National Fellowship and UGC-CAS. The work of S. Hazra was supported through Research Fellowships of CSIR, IISc and a Post-doctoral Fellowship of the NBHM}

\begin{abstract}

A bounded linear operator $T$ on a Hilbert space is said to be homogeneous 
if $\varphi(T)$ is unitarily equivalent to $T$ for all $\varphi$ in the group 
M\"{o}b of bi-holomorphic automorphisms of the unit disc.
A projective unitary representation $\sigma$ of M\"{o}b 
is said to be associated with an operator T if 
$\varphi(T)= \sigma(\varphi)^\star T \sigma(\varphi)$ for all $\varphi$  in M\"{o}b.  

In this paper, we develop a M\"{o}bius equivariant version of the  Sz.-Nagy--Foias model theory for completely non-unitary (cnu) contractions. As an application, we prove that if T is a cnu contraction with associated (projective 
unitary) representation $\sigma$, then there is a unique projective unitary 
representation $\hat{\sigma}$, extending $\sigma$, associated with the minimal unitary dilation  of $T$. The representation $\hat{\sigma}$ is given in terms of $\sigma$ by the formula 
$$
\hat{\sigma} = (\pi \otimes D_1^+) \oplus \sigma \oplus (\pi_\star \otimes D_1^-),
$$
where $D_1^\pm$  are the two Discrete series representations (one holomorphic and the other anti-holomorphic) living on the Hardy space $H^2(\mathbb D)$, and $\pi, \pi_\star$  are representations of M\"{o}b living on the two defect spaces of $T$  defined explicitly in terms of $\sigma$.  

Moreover, a cnu contraction $T$ has an associated representation
if and only if its Sz.-Nagy--Foias characteristic function $\theta_T$ has the product form 
$\theta_T(z) = \pi_\star(\varphi_z)^* \theta_T(0) \pi(\varphi_z),$ $z\in \mathbb D$, where $\varphi_z$ is the involution in M\"{o}b mapping $z$ to $0.$ 
We obtain a concrete realization of this product formula 
for a large subclass of homogeneous cnu contractions from the Cowen-Douglas class. 
\end{abstract}

\maketitle 

\section{Introduction}

All Hilbert spaces in this paper are complex and separable. All operators are
linear and bounded operators between Hilbert spaces. For any two Hilbert spaces $\mathcal H$ and $\mathcal K$, $\mathcal B(\mathcal H, \mathcal K)$ denotes the Banach space of all operators from $\mathcal H$ to $\mathcal K$.  We shall abbreviate  $\mathcal B(\mathcal H, \mathcal H)$ to $\mathcal B(\mathcal H)$. 
The (M\"{o}bius) group of all
bi-holomorphic self-maps of the unit disc $\mathbb{D}$ (in the complex plane
$\mathbb{C}$) shall be denoted by M\"{o}b.  As a topological group (with the
topology of locally uniform convergence) it is isomorphic to PSU($1,1$) and to
PSL($2,\mathbb{R}$). 

Recall from \cite{Mi,CM,con} that an operator $T$  from a Hilbert space into
itself is said to be  {\sf \emph{homogeneous}} if $\varphi(T)$  is unitarily
equivalent to $T$  for all $\varphi$ in M\"ob which are analytic in a neighbourhood of the
spectrum of  $T$.  It  was shown in  \cite{con} that the spectrum of a
homogeneous operator $T$ is either the unit circle $\mathbb{T}$ or the closed
unit disc $\overline{\mathbb{D}}$, so that, actually, $\varphi(T)$ is
unitarily equivalent to $T$ for all $\varphi$ in M\"{o}b. 
Recall that (see \cite{Mackey,KRP,Va} for instance) a {\sf \emph{projective unitary  representation}} $\sigma$ of M\"{o}b on a Hilbert space $\mathcal H$ is a Borel function $\sigma:\mbox{\rm M\"{o}b} \to \mathcal U(\mathcal H)$, satisfying $\sigma(\rm{id}) = I$, for which there is a function $m:\mbox{\rm M\"{o}b} \times \mbox{\rm M\"{o}b} \to \mathbb T$ satisfying 
\begin{equation} \label{reprep}
\sigma(\varphi_1 \varphi_2)=  m(\varphi_1, \varphi_2) \sigma(\varphi_1) \sigma(\varphi_2), \,\,\varphi_1,\varphi_2 \in \mbox{\rm M\"{o}b}. 
\end{equation}
Here $\mathcal U(\mathcal H)$ is the topological group of unitary operators on the Hilbert space $\mathcal H$. 
Clearly $\sigma$ determines the function $m$  by the Equation \eqref{reprep} and $m$ is Borel. This function is called the {\sf\emph {multiplier}} 
of $\sigma$.  Clearly, $m(\varphi, {\rm id}) = 1 = m({\rm id}, \varphi)$.  Evaluating $\sigma(\varphi_1 \varphi_2 \varphi_3)$ in two different ways, one sees that the multiplier $m$ of any projective unitary representation satisfies the identity
\begin{equation} \label{mult}
m(\varphi_1, \varphi_2) m(\varphi_1 \varphi_2, \varphi_3) = m(\varphi_1, \varphi_2 \varphi_3) m(\varphi_2,\varphi_3), \,\,\varphi_1, \varphi_2,\varphi_3 \in \mbox{\rm M\"{o}b}.
\end{equation} 

Two projective unitary representations $\sigma, \tilde{\sigma}$ of 
M\"{o}b living on the respective Hilbert spaces $\mathcal H$, $\widetilde{\mathcal H}$ are said to be {\sf \emph{equivalent}} if there is a unitary $U:\mathcal H \to \widetilde{\mathcal H}$ and  a Borel function $f:\mbox{M\"{o}b} \to \mathbb T$ such that $\tilde{\sigma}(\varphi) = f(\varphi) U \sigma(\varphi) U^*$ for all $\varphi$ in M\"{o}b.  In this paper, by the word ``representation'', we always mean a projective unitary representation of M\"{o}b.

We say that a projective unitary representation $\sigma$ of M\"{o}b is {\sf \emph{associated}}  with an operator $T$ if
\[
\varphi(T)=\sigma(\varphi)^{\ast}T\sigma(\varphi)
\]
for all $\varphi$ in M\"{o}b.  We say that an
operator from a Hilbert space into itself is an {\sf \emph{associator}}   if there is a projective unitary representation of \Mob associated with it. Clearly, all associators are homogeneous, though the converse is not true.
However in \cite[Theorem 2.2]{shift} it is shown that, conversely, each
{\sf \emph{irreducible}}  homogeneous operator is an associator (see also \cite{Mi-Sa}), and -- further --
its associated representation is unique up to equivalence.


A huge number of (unitarily inequivalent) examples of homogeneous operators
are known. (See the survey article  \cite{survey} as well as the more recent papers \cite{shift,KM,ClassCD}).  Since the direct sum (more generally direct integral) of homogeneous operators
is again homogeneous, a natural problem is the classification (up to unitary equivalence) of {\sf \emph{atomic homogeneous operators}}, that is, those homogeneous operators which can not
be written as the direct sum of two homogeneous operators. In this generality,
this problem remains unsolved. A beginning in this direction was made in \cite{shift} 
where we classified the homogeneous scalar weighted shifts. Moreover, all the homogeneous operators in the Cowen-Douglas class have been described modulo unitary equivalence in the paper \cite{ClassCD}. Clearly,
irreducible homogeneous operators are atomic. In this connection, it is
amusing to note that we know of only two examples of atomic homogeneous
operators which are not irreducible. These are the multiplication operators --
by the respective co-ordinate functions -- on the Hilbert spaces
$L^{2}(\mathbb{T)}$ and $L^{2}(\mathbb{D)}$. Both of these examples happen to
be normal operators. We do not know if all atomic
homogeneous operators are associators.

Recall that an operator $T$ is said to be a contraction if $\left\|
T\right\|  \leq1$. The objective of this paper is to set up a theoretical framework for the eventual classification of all cnu contractive associators. This is achieved by an application of the  Sz.-Nagy--Foias theory \cite{Na-Fo} of {\sf \emph{cnu}} (completely non-unitary)
contractions. A contraction is said to be cnu if it has no unitary part
(that is, if it cannot be written as the direct sum of two operators one of which
is unitary). A contraction $T$ is said to be {\sf \emph{pure}} if $\left\|
Tx\right\|  <\left\|  x\right\|  $ for all non-zero vectors $x$. The
afore-mentioned theory \ attaches to any cnu contraction $T$ a pure
contraction valued analytic function on $\mathbb{D}$, called the
{\sf \emph{characteristic function}} of $T$. Two cnu contractions are unitarily
equivalent iff their characteristic functions $\theta_{1}$ and $\theta_{2}$
{\sf \emph{coincide}}  (that is, if and only if there exist two unitaries $u$ and $v$ such that
$u\theta_{1}(z)=\theta_{2}(z)v$ for all $z\in\mathbb{D}$). In Section 2 of
this paper, we briefly review this theory, mostly following Nikolski in \cite{Nik}, but with some twists of our own. With any cnu contraction, we begin by
associating its {\sf \emph{characteristic operator.}}  It carries exactly the same
information as the characteristic function; in fact it is easy to obtain one
in terms of the other. But it is the characteristic operator which emerges
most naturally from the study of {\sf \emph{minimal unitary (power) dilations.}}
Of course, this notion was always implicit in the theory -- we find it
convenient to make it explicit. Another innovation is to emphasize the natural
relationship between the Sz.-Nagy--Foias theory and the M\"{o}bius group. This is the content of Section 3.  It is surprising that
the role of the M\"{o}bius group in the Sz.-Nagy--Foias theory was never made
explicit nor was it used to its full potential. 

In Section 4 of this paper we exploit this relationship to prove that
if $T$ is a cnu contraction with associated (projective unitary)
representation $\sigma$, then there is a unique projective unitary
representation $\hat{\sigma}$, extending $\sigma$, associated with the
minimal unitary dilation $W$ of $T$. Its existence is a theorem from \cite{CM},
while its uniqueness is quite easy to establish. What is surprising is that we
are able to write an explicit and pretty formula for $\hat{\sigma}$ in
terms of $\sigma:$%
\[
\hat{\sigma}=(\pi\otimes D_{1}^{+})\oplus\sigma\oplus(\pi_{\ast}\otimes D_{1}^{-}),
\]
where $D_{1}^{\pm}$ are the two Discrete series representations (one
holomorphic and the other anti-holomorphic) living on the Hardy space
$H^{2}(\mathbb D)$, and $\pi$ and $\pi_{\ast}$ are representations of the
M\"{o}bius group living on the two defect spaces of $T$ and explicitly defined in
terms of $\sigma$ by the rather mysterious formulae presented in  Theorem \ref{tworeps}.  In the language of
Mackey (see \cite{Mackey}) the triple $(W,\hat{\sigma},\mathbb{T})$ is a system of
imprimitivity of the M\"{o}bius group.  The imprimitivity relationship, in this
case, is just the condition imposed by homogeniety on the operator $W$.  Thus
the study of cnu contractive associators via their minimal unitary
dilation is an equivariant version of the model theory for contractions
developed by Sz.-Nagy--Foias. 

In continuation of these ideas, we obtain a characterization of cnu contractive associators, first in terms of their characteristic operators,
and eventually in terms of their characteristic functions. This leads to a
pleasant {\sf \emph{product formula}} for the characteristic function $\theta$ of
any cnu contractive associator:
\begin{equation}\label{product}
\theta(z)=\pi_{\ast}(\varphi_{z})^{\ast}C\pi(\varphi_{z}),\quad
z\in\mathbb{D},
\end{equation}
where  for $z\in \mathbb{D}$,
$\varphi_{z}$ is the unique involution in \Mob which interchanges $0$ and $z$.
\ Also, $\pi,\pi_{\ast}$ are two projective representations of \Mob (living on
the defect spaces of $T$) with identical multipliers. The operator $C$ is a
pure contraction (from the space of $\pi$ to the space of $\pi_{\ast}$)
intertwining the restricted representations $\pi|_{\mathbb{K}}$ and $\pi
_{\ast}|_{\mathbb{K}}$ of the maximal compact subgroup $\mathbb{K}$ of \Mob: $\mathbb{K}=\{\varphi\in$\Mob $:\varphi(0)=0\}.$

The representations $\pi$ and $\pi_{\ast}$ of \Mob which
occur in the product formula are precisely the same representations which occur in
the above description of $\hat{\sigma}$. In view of the product formula,
we refer to the representations $\pi$ and $\pi_{\ast}$ as the {\sf \emph{ (right
and left) companions}} of the operator with characteristic function $\theta$.
Notice that they are {\sf \emph{compatible}} in the sense of having identical
multipliers. (Mutual compatibility is the obvious necessary and sufficient
condition on a family of projective representations for its direct integral to
define a projective representation.) As a converse, we show that whenever
$\pi$ and $\pi_{\ast}$ are two compatible projective unitary representations
of \Mob and $C:\mathcal{H}_{\pi}\longrightarrow\mathcal{H}_{\pi_{\ast}}$ is a
pure contraction intertwining their restrictions to $\mathbb{K}$, the function
$\theta$ defined by the product formula \eqref{product} is a {\sf \emph{homogeneous
characteristic function}} (that is, the characteristic function of a 
cnu contractive associator) -- provided, of course, that $\theta$ is analytic. Thus, within the class of cnu contractions, the associators are characterized by the presence of such a product formula. 

In the penultimate section, we find explicit product formulae for the characteristic functions of most of the irreducible homogeneous contractions in the Cowen-Douglas class $B_n(\mathbb D)$ whose associated representation is multiplicity free. In the final section, we present a similar description for an extremal family in this class and state a conjectural complete description for the entire class. 

In consequence of the results of this paper, the problem of classification of cnu contractive associators  boils down to the following question. For any two compatible projective unitary representations $\pi$ and $\pi_*$ of M\"{o}b, living on the Hilbert spaces $\mathcal{H}$ and $\mathcal{H}_*$, let $V(\pi, \pi_*)$ denote the Banach space of all bounded operators $C: \mathcal{H} \to \mathcal{H}_*$ such that $C$ intertwines $\pi_{|\mathbb{K}}$ and ${\pi_*}_{|\mathbb{K}}$, and the function $z \mapsto \pi_*(\varphi_z)^*C\pi(\varphi_z)$ is holomorphic on $\mathbb{D}$. Determine $V(\pi, \pi_*)$ and find all pure contractions in this space.


\section{Sz.-Nagy--Foias Theory.}

In this section we provide a convenient summary of the theory of cnu
contractions, their dilations and characteristic functions. This summary largely follows the exposition in \cite{Nik}.

\subsection{Minimal power dilations -- isometric, co-isometric and unitary.}

If $\mathcal{H}$ is a Hilbert subspace of a Hilbert space $\mathcal{K}$,  
$i:\mathcal{H\longrightarrow K}$ is the inclusion map and
$X:\mathcal{K\longrightarrow K}$ is an operator, then $T:=i^{\ast}Xi: \mathcal{H\longmapsto H}$ is called the { compression} of $X$ to
$\mathcal{H}$ and $X$ is called a \emph{dilation }of $T$. If, further, we
have $p(T)=i^{\ast}p(X)i$ for all polynomials $p\in\mathbb{C}[z]$, then $T$ is
called the \emph{power compression} of $X$ to $\mathcal{H}$ and $X$ is
called a \emph{power dilation} of $T$. A famous lemma of Sarason (cf. \cite{Nik})
says that the dilation $X$ of $T$ is a power dilation if and only if $\mathcal{H}$ is a
semi-invariant subspace (that is, the intersection of an invariant subspace with
a co-invariant subspace) for $X$. If the power dilation $X$ of $T$ is an
isometry/co-isometry/unitary then it is called an isometric/co-isometric/unitary power dilation. If, further, there is no Hilbert space
$\mathcal{K}_{0}$ with $\mathcal{H\subseteq K}_{0}\subset\mathcal{K}$ such
that the compression of $X$ to $\mathcal{K}_{0}$ is an isometric /
co-isometric / unitary power dilation of $T$ then $X$ is called a
{\sf \emph{minimal isometric{\rm /} co-isometric {\rm /} unitary power dilation}} of
$T$. Obviously $\left\|  T\right\|  \leq\left\|  X\right\|  $ for any (power)
dilation $X$ of $T$. Thus, for the existence of an isometric/co-isometric/unitary 
power dilation of $T$, $T$ must be a contraction.  A basic result due
to Sz.-Nagy says that any contraction has a minimal isometric/co-isometric/
unitary power dilation, and it is essentially unique. We proceed to elaborate.
However, in anticipation of this result, we shall use the definite article
`the' when talking of these minimal dilations. Also note that $X$ is the
minimal co-isometric dilation of $T$  if and only if $X^{\ast}$ is the minimal isometric
dilation of $T^{\ast}$. Thus, for most purposes, it suffices to look at the
minimal isometric and unitary dilations.

The following lemma clearly includes the uniqueness (though not existence!) of
the minimal unitary (or isometric, or co-isometric) dilation. We shall need the
 full strength of this lemma in the next section.

\begin{lemma}
\label{lifting} For $i=1,2$, let $T_{i}:\mathcal{H}_{i}\longrightarrow
\mathcal{H}_{i}$ be Hilbert space contractions with corresponding minimal
isometric/co-isometric/unitary power dilations  $\widehat{T}_{i}:\widehat{\mathcal{H}}_{i}\longrightarrow\widehat{\mathcal{H}}_{i}$. Let
$U:\mathcal{H}_{1}\longrightarrow\mathcal{H}_{2}$ be a unitary such that
$UT_{1}=T_{2}U$. Then there is a unique unitary $\widehat{U}:\widehat
{\mathcal{H}}_{1}\longrightarrow\widehat{\mathcal{H}}_{2}$ such that
$\widehat{U}|_{\mathcal{H}_{1}}=U$ and $\widehat{U}\widehat{T}_{1}%
=\widehat{T}_{2}\widehat{U}$.
\end{lemma}

\begin{proof} First assume that $\widehat{T}_i$ are isometric dilations.
We begin by proving the uniqueness. To this end, we claim that any unitary
$\widehat{U}$ as in the conclusion of this lemma satisfies, for $n=0,1,2,...$%

\begin{equation}\label{damn}
\widehat{U}(\widehat{T}_{1}^{n}x)=\widehat{T}_{2}^{n}
(Ux)\quad\text{for all }x\in\mathcal{H}_{1}.
\end{equation}
This is easily proved by induction on $n$.
 Notice that this part of the proof does not make use
of the minimality of the dilations.

Now suppose $\widehat{T_{i}}$ are minimal isometric dilations.
Then  the closed linear span of the set
$$
A_i := \{\widehat{T_{i}
}^{n}x:n=0,1,2,...,\,x\in\mathcal{H}_{i}\}\subset\widehat{\mathcal{H}}_{i}
$$
 is invariant under $\widehat{T_{i}}$ and contains
$\mathcal{H}_{i}$. The restriction of $\widehat{T_i}$ to this
invariant subspace is an isometric power dilation of $T_i$.
Therefore, the minimality of the isometric dilation implies that
$A_i$  is a total set in $\mathcal{H}_i$. Since \eqref{damn} gives
the value of $\widehat{U}$ on $A_1$, $\widehat{U}$ is uniquely
determined. To show existence of the unitary $\widehat{U}$, \ we
verify that $\widehat{U}$, defined on the  total set $A_1$ by \eqref{damn}
preserves the inner product and maps $A_1$ onto the total set $A_2$,
it then  extends to a well-defined unitary. 

Next suppose  $\widehat{T_{i}}$ is a minimal co-isometric dilation of
$T_{i}$ for $i=1,2$. Then $\widehat{T_{i}}^{\ast}$ is a minimal isometric
dilation of $\ T_{i}^{\ast}$ and $U$ intertwines $T_{1}^{\ast}$ and
$T_{2}^{\ast}.$ Therefore, by the previous part of this result, there is a
unique unitary $\widehat{U}$ which intertwines $\widehat{T}_{1}^{\ast}$ and
$\widehat{T}_{2}^{\ast}$ and extends $U$. Then $\widehat{U}$ intertwines
$\widehat{T}_{1}$ and  $\widehat{T}_{2}$. This completes the proof in the
case of co-isometric dilations. Notice that, in this case, the equation \eqref{damn} applied to $T_i^*$ shows that $\widehat{U}$ satisfies, for $n=0,1,2, \ldots,$
\begin{equation} \label{damn*}
\widehat{U}(\widehat{T}_1^{*n}x)=\widehat{T}_2^{*n}(Ux)  \quad\text{for all }x\in\mathcal{H}_{1},
\end{equation}
when $\widehat{T}_i$ are co-isometric dilations. Again, minimality of the dilations is not necessary for the validity of this equation.

Finally, let  \ $\widehat{T_{i}}$ \ be a minimal unitary dilation of
$T_{i}$ for $i=1,2$. By the observation above, we now have, for all $n=0, \pm1,\pm2, \ldots,$
\begin{equation} \label{ddamn}
 \widehat{U}(\widehat{T}_1^{*n}x)=\widehat{T}_2^{*n}(Ux)  \quad\text{for all }x\in\mathcal{H}_{1},
 \end{equation}
  Again, minimality of $\widehat{T}_i$ shows that the set $B_i :=
 \{\widehat{T_i}^nx : x \in {\mathcal H_i}, n = 0, \pm 1, \pm 2, \ldots \}$
 is a total set in $\widehat{\mathcal H}_i $ and $\widehat{U}$ is determined 
 on the total set $B_1$ by the equation \eqref{ddamn}. This proves uniqueness. 
 Also, $\widehat{U}$ preserves the inner product and maps $B_1$ 
 onto $B_2$. Therefore it extends to a unitary on $\widehat{\mathcal H}_1$, 
 proving existence.  
\end{proof}

\begin{notation} \label{nota1}
(a) Let $H^2 = H^2(\D)$ be the usual Hilbert space of analytic functions on the open unit disc
  $\D$ with square integrable boundary value (radial limit). Also, $S : H^2 \rightarrow H^2$ be the
     standard (un-weighted) unilateral shift (given by $(Sf)(z)=zf(z)$).

(b) For any contraction $T : \mathcal{H} \rightarrow \mathcal{H}$, let $\mathcal{D}$ (respectively
$\mathcal{D}_{\ast}$) denote the closure of the range of $(I-T^*T)^{1/2}$ (respectively, of 
$(I-TT^*)^{1/2}$). Let  $D :\mathcal{H} \rightarrow \mathcal{D}$ (respectively $D_\ast :
\mathcal{H} \rightarrow \mathcal{D}_\ast $)  be the operator given by $x \mapsto (I-T^*T)^{1/2}x$
(respectively $x \mapsto (I - TT^*)^{1/2}x$).
  Thus, by definition, $D$ and $D_\ast$ are  contractions with  dense range. These are called the 
  defect operators of $T$, and their co-domains $\mathcal{D}$ and $\mathcal{D}_\ast$ are called 
  the defect spaces of $T$.  In the existing literature, it is customary to indicate the dependence
   of these defect operators and spaces on the initial  contraction $T$ by means of a suffix in their 
   names. We have departed from this established practice for typographical and esthetic reasons. 
   We hope this will not cause any confusion and, in each case, the initial contraction will be 
   clear from the context.

(c)  We shall identify the tensor product Hilbert space $ \mathcal{D} \otimes H^2$ with the Hilbert
    space of $\mathcal{D}$-valued analytic functions  on $\D$ with square integrable boundary value
     (via the usual identification of $v \otimes f$ with the function $z \mapsto f(z)v$).   Thus
      $I\otimes S : \mathcal{D} \otimes H^2 \rightarrow \mathcal{D}\otimes H^2$ is the unilateral shift 
      of multiplicity dim($\mathcal{D}$) -- it is formally given by the same formula as $S$, 
      when its domain is viewed as a space of $\mathcal{D}$-valued functions. Likewise,
       $\mathcal{D}_\ast \otimes H^2$ is viewed as a Hilbert space of $\mathcal{D}_\ast$-valued 
       analytic functions and $I \otimes S^*$ is to be viewed as a backward shift (with multiplicity)
        on this space.

   (d)    Finally we let  $\boldsymbol i : \mathcal{D} \rightarrow \mathcal{D} \otimes H^2$
   (respectively $ \boldsymbol i_\ast : \mathcal{D}_\ast \rightarrow \mathcal{D}_\ast \otimes H^2$) 
   be the `inclusion' maps given by the formula $x \mapsto x \otimes \mathbf{1}$. (Here, of course,
    $\mathbf{1}$ is the constant function $1$ in $H^2$.)  
  \end{notation}

The following result is, of course, well known. We include its proof for completeness and to 
ease the development of related ideas.

\begin{theorem} \label{isodil} Every contraction $T$ on a Hilbert space $\mathcal{H}$
has a minimal isometric (or co-isometric) dilation. It is unique upto unitary equivalence 
(via a unitary which leaves the subspace $\mathcal{H}$ invariant
 and restricts  to the identity operator on this subspace).
\end{theorem}

\begin{proof} The uniqueness is immediate from Lemma \ref{lifting} with $T_1 =T_2 = T$, $\mathcal{H}_1 =
\mathcal{H}_2=\mathcal{H}$ and $U=I$. If we prove the existence result in the isometric case 
then the co-isometric case follows by applying this result to the contraction $T^*$. Thus it
 suffices to prove the existence of a minimal isometric dilation for $T$. Let $\tilde{T}$ be the operator
 on the Hilbert space $\mathcal{K} := (\mathcal{D} \otimes H^2)\oplus \mathcal{H}$ given by 
  \begin{equation} \label{isodilform} \tilde{T}  = \left ( \begin{array}{cc} I \otimes S & iD \\ 0 & T
 \end{array} \right ). \end{equation}
 
We claim that $\tilde{T}$ is the minimal isometric dilation of $T$. Since $S^*S=I$ and $D^*D+T^*T= I_{\mathcal{H}}$
  and the range closure of $D$ (viewed as an operator into $\mathcal{D} \otimes H^2$, see Notation  
  \ref{nota1}) equals the kernel of $I \otimes S$, it follows that $\tilde{T}^*\tilde{T}=I_{\mathcal{K}}$. 
 Thus $\tilde{T}$ is an isometry. Since $T$ is a diagonal entry of the upper triangular (block) matrix $\tilde{T}$,
  it follows that $\tilde{T}$ is a (power) dilation of $T$. (Note that $\mathcal{H} = 0 \oplus \mathcal{H}$ 
  is co-invariant under $\tilde{T}$.) To show that $\tilde{T}$ is minimal, let $\mathcal{K}_0$ be the closed linear 
  span in $\mathcal{K}$ of the set $\{ \tilde{T}^nx : x \in \mathcal{H}, \; n=0,1,2, \ldots \}$. We need to 
  show that $\mathcal{K}_0=\mathcal{K}$. Let $\{ e_n : n=0,1,2,\ldots \}$ be the standard orthonormal
   basis of $H^2$. (Thus $e_n(z) =z^n, \;\; z \in \D$.) Clearly it suffices to show that
    $v \otimes e_n \in \mathcal{K}_0$ for all $v \in \mathcal{D}$ and all $n \geq 0$. We do this by 
    induction on $n$. Trivially, this is true for $n=0$ since $\mathcal{D} \otimes e_0$ is the range 
    closure of $D$. This starts the induction.
    
    Note that, for non-commuting variables $a,b,c$, we have :
      $$ \left ( \begin{array}{cc} a & b \\ 0 & c \end{array} \right )^n = \left ( \begin{array}{cc}
    a^n & \sum_{h=0}^{n-1} a^hbc^{n-1-h} \\ 0 & c^n \end{array} \right ), $$
        as may be seen by induction on $n$.
    
Therefore,
    $$\tilde{T}^{n+1} \left(\begin{array}{c} 0 \\ x \end{array} \right ) = \begin{pmatrix} \sum_{h=0}^n 
    (I \otimes S^h)DT^{n-h}x \\ T^{n+1}x \end{pmatrix} \in \mathcal{K}_0. $$
        Since, clearly, $\mathcal{K}_0 \supseteq \mathcal{H}$, it follows that
    $$ \sum_{h=0}^n (I \otimes S^h)DT^{n-h}x = \sum_{h=0}^n (DT^{n-h}x) \otimes e_h \in \mathcal{K}_0. $$
    But, by induction hypothesis, $\sum_{h=0}^{n-1} (DT^{n-h}x) \otimes e_h \in \mathcal{K}_0$.
     Subtracting, we get $(Dx)\otimes e_n \in \mathcal{K}_0$ for all $x$. Since $D$ has dense
      range in $\mathcal{D}$, it follows that $v \otimes e_n \in \mathcal{K}_0$ for all $v \in \mathcal{D}$.
       This completes the induction. Thus $\mathcal{D} \otimes H^2 \subseteq \mathcal{K}_0$.
        Since also, $\mathcal{H} \subseteq \mathcal{K}_0$, it follows that $\mathcal{K}_0 = \mathcal{K}$. 
      Thus, the operator $\tilde{T}$ defined by Equation \eqref{isodilform} is indeed the unique 
      minimal isometric dilation of $T$. Since the minimal co-isometric dilation of $T$ is the adjoint
       of the minimal isometric dilation of $T^*$, it follows that the unique minimal co-isometric 
       dilation $\tilde{T}_\ast$
       of $T$ is given by the formula
         \begin{equation} \label{coisodilform} \tilde{T}_\ast = \begin{pmatrix}T & (i_\ast D_\ast)^* \\ 0 & I \otimes S^* \end{pmatrix},
       \end{equation}
       acting on the Hilbert space $\mathcal{H} \oplus (\mathcal{D}_\ast \otimes H^2)$. 
\end{proof}

\begin{theorem} \label{unidil} Every contraction $T$ on a Hilbert space $\mathcal{H}$ has a minimal unitary dilation which is unique upto unitary equivalence. Explicitly, it is the unitary $\widehat{T}$ on the Hilbert space $(\mathcal{D}
\otimes H^2) \oplus \mathcal{H} \oplus (\mathcal{D}_\ast \otimes H^2)$  given by the formula
 $$ \widehat{T} = \begin{pmatrix} I \otimes S & iD & i\,C^*\,i_\ast^* \\ 0 & T & (i_\ast D_\ast)^* \\ 0 & 0 & I \otimes S^* 
 \end{pmatrix}, $$
where $C : \mathcal{D} \rightarrow \mathcal{D}_\ast$ is the operator $x \mapsto -Tx$. \end{theorem}

(Because of the well known identity $ T(I-T^*T)^{1/2} = (I-TT^*)^{1/2}T$, $T$ maps $\mathcal{D}$ into $\mathcal{D}_\ast$. Thus $C$ indeed maps $\mathcal{D}$ into $\mathcal{D}_\ast$ and satisfies $CD=-D_\ast T$. The operators $D,\: D_\ast$ and spaces $\mathcal{D},\: \mathcal{D}_\ast$ are as in Notation \ref{nota1}.) 

\begin{proof} Again, uniqueness follows from  Lemma \ref{lifting}. To prove existence, let $\widehat{T}$ be a minimal unitary dilation of $T$. Clearly the compressions of $\widehat{T}$ to the subspaces generated by the vectors $\{\widehat{T}^nx \,|\, n \geq 0, x \in \mathcal{H}\}$ and $\{\widehat{T}^nx \,|\, n \leq 0, x \in \mathcal{H} \}$  are the minimal isometric and co-isometric dilations of $T$. Therefore, by the above, these compressions may be identified with the operators $\tilde{T}$ and $\tilde{T}_\ast$ (given by the formulae \eqref{isodilform} and \eqref{coisodilform}) and consequently these two subspaces are identified with $(\mathcal{D} \otimes H^2) \oplus \mathcal{H}$ and $\mathcal{H} \oplus (\mathcal{D}_\ast \otimes H^2)$. Thus, $\widehat{T}$ acts on $(\mathcal{D} \oplus H^2) \oplus \mathcal{H} \oplus (\mathcal{D}_\ast \otimes 
H^2)$ and is given by the formula in the statement of the theorem, except that the $(1,3)$-entry of this block operator has an unknown entry $A : \mathcal{D}_\ast \otimes H^2 \to \mathcal{D} \otimes H^2$. Now, since $\widehat{T}^*\widehat{T}=I$, equating    entries of this `matrix equation', we find that  $A$ must satisfy  (i) $(I \otimes S^*)A=0$  (ii) $(iD)^*A = - (i_\ast D_\ast T)^*$, and (iii) $ A^*A=(I-D_\ast D_\ast^*) \otimes P$. (Here $P=I-SS^* : H^2 \to H^2$ is the orthogonal projection $Pf =f(0)\mathbf{1}$.)
 Now, (i) says that $\rm{ker}(A^*) \supseteq (\mathcal{D} \otimes \mathbf{1})^\perp=\mathcal{D} \otimes \rm{range}(I-P)$ and (ii) says that $A^*(Dx \otimes \mathbf{1})= -(D_\ast Tx) \otimes \mathbf{1}$. This determines $A^*$ on all pure tensors $v \otimes f$ and hence it is determined throughout $\mathcal{D}_\ast \otimes H^2$ : 
 $$ A^*(Dx \otimes f) = A^*(Dx \otimes (I-P)f) + A^*(Dx \otimes Pf) = -(D_\ast Tx) \otimes \mathbf{1}. $$
 Thus $A$ is determined by the requirements (i) and (ii). Since one readily verifies that $iC^*i_\ast ^*$ satisfies (i) and (ii), it follows that we must have $A= iC^*i_\ast^*$. It is now easy to see that this choice of $A$ satisfies (iii) as well, so that $\widehat{T}$ given above is an isometry. Since $\widehat{T}^*$ is obtained from this formula for $\widehat{T}$ by replacing $T$ by $T^*$ (and, consequently, replacing $\mathcal{D}$ by $\mathcal{D}_\ast$ and so on) it follows that $\widehat{T}$ is also a co-isometry. 
 Therefore, $\widehat{T}$ is a unitary.
\end{proof}

\subsection{Characteristic Operators and Characteristic Functions}

We continue with the set-up introduced above. Thus $T$ is a contraction on a Hilbert space $\mathcal{H}$ with defect spaces $\mathcal{D}$, $\mathcal{D}_\ast$, and defect operators $D$, $D_\ast$. The minimal unitary dilation $\widehat{T}$ of $T$ lives on the space $\widehat{\mathcal{H}} := (\mathcal{D} \otimes H^2) \oplus \mathcal{H} \oplus (\mathcal{D}_\ast \otimes H^2)$ and is given explicitly as in Theorem \ref{unidil}.  From this description, one sees that there is a `visible' copy  
$$\mathcal{F} := (\mathcal{D} \otimes H^2) \oplus 0 \oplus 0$$
of $\mathcal{D} \otimes H^2$ inside the dilation space $\widehat{\mathcal{H}}$. It is invariant under the dilation operator $\widehat{T}$, and the restriction of $\widehat{T}$ to this subspace is a copy of the unilateral shift of multiplicity $\rm{dim}(\mathcal{D})$. It turns out that there is also an `invisible' copy $\mathcal{F}_\ast$ of $\mathcal{D}_\ast \otimes H^2$ inside $\widehat{\mathcal{H}}$ which is also invariant under $\widehat{T}$ and such that the restriction of $\widehat{T}$ to $\mathcal{F}_\ast$ is a copy of the unilateral shift of multiplicity $\rm{dim}( \mathcal{D}_\ast)$. (The visible copy of $\mathcal{D}_\ast \otimes H^2$ inside the dilation space is co-variant under $\widehat{T}$.)
Namely, we have :
$$ \mathcal{F}_\ast := \bigoplus_{n=0}^\infty \widehat{T}^n (\mathcal{D}_\ast \otimes \mathbf{1}).$$
Since $\widehat{T}$ is a unitary, it follows that  $\widehat{T}^{*m}(\mathcal D_\ast\otimes 1) \perp \mathcal D_\ast\otimes 1,$ $m > 0$, and therefore the sum  is an orthogonal direct sum, clearly invariant under $\widehat{T}$. 
If we define  $\Psi :\mathcal{D}_\ast \otimes H^2 \to \mathcal{F}_\ast$ by 
$$ \Psi (v \otimes e_n) = \widehat{T}^n (v), \: n=0,1,2, \ldots, \: v \in \mathcal{D}_\ast,$$
where $e_n, \; n=0,1,2, \ldots $ is the standard orthonormal basis of $H^2$ (thus $e_0=\mathbf{1}$ and $e_n = Se_{n-1}$ for $n \geq 1$), then it immediately follows that $\Psi$ is  a unitary which intertwines $I \otimes S$ with $\widehat{T}|_{\mathcal{F}_\ast }$.

Now, the {\em characteristic operator} $\Theta$ of the contraction $T$ is defined to be the `part' of $\widehat{T}^*$ which goes from $\mathcal{F}$ to $\mathcal{F}_\ast$. That is, 
$$ \Theta := \mathbf{j}_\ast^* \:\widehat{T}^*\: \mathbf{j}, $$
where $\mathbf{j} : \mathcal{F} \to \widehat{\mathcal{H}}$ and $\mathbf{j}_\ast : \mathcal{F}_\ast \to \widehat{\mathcal{H}}$ are the respective inclusion maps. 

We use the unitary $\Psi$ to identify $\mathcal{F}_\ast $ with $\mathcal{D}_\ast \otimes H^2$. After this identification, we have $\Theta : \mathcal{D} \otimes H^2 \to \mathcal{D}_\ast \otimes H^2$. Let's calculate the characteristic operator $\Theta$ explicitly.  We need the following formula for the restriction of the projection $\mathbf{j}_\ast^*$ to $\mathcal{H}$ :

\begin{lemma} \label{project} For $x \in \mathcal{H}$, we have $\mathbf{j}_\ast^*(x)= \sum_{n=1}^\infty \widehat{T}^n (D_\ast {T^*}^{n-1}x \otimes \mathbf{1})$.
\end{lemma}

\begin{proof} Since $\mathbf{j}_\ast^*$ is the orthogonal projection onto $\mathcal{F}_\ast = \bigoplus_{n=0}^\infty \widehat{T}^n (\mathcal{D}_\ast \otimes \mathbf{1})$ (orthogonal direct sum) and $x \in \mathcal{H}$ is orthogonal to $\mathcal{D} \otimes \mathbf{1}$, it follows that $\mathbf{j}_\ast^*(x) = \sum_{n=1}^\infty \widehat{T}^n (\alpha_n (x) \otimes \mathbf{1}).$ 
Here, the vectors $\alpha_n(x) \in \mathcal{D}_\ast$ are uniquely determined by the requirement $x - \widehat{T}^n (\alpha_n(x) \otimes \mathbf{1}) \perp \widehat{T}^n(\alpha_n(x)\otimes \mathbf{1})$. So, we need to show that $\alpha_n(x)=D_\ast {T^*}^{n-1}x$ for $n \geq 1$. Clearly, it suffices to show that  $\alpha_1(x) = D_\ast x$ and $\alpha_{n+1}(x) = \alpha_n (T^*x)$ for $n \geq 1$.

Now, using the explicit form of $\widehat{T}$ given in Theorem \ref{unidil}, we compute :

\begin{eqnarray*} \lefteqn{\langle x - \widehat{T}(D_\ast x \otimes \mathbf{1}),  \widehat{T}(D_\ast x \otimes \mathbf{1}) \rangle}\hspace{3em}\\  
& = & 
\langle (-C^*D_\ast x \otimes \mathbf{1}) \oplus (I-D_\ast^* D_\ast)x, (C^*D_\ast x \otimes \mathbf{1} \oplus D_\ast^* D_\ast x \rangle \\ & = & -\langle C^*D_\ast x, C^*D_\ast x\rangle + \langle (I-D_\ast^* D_\ast) x, D_\ast^* D_\ast x \rangle \\
& = & - \langle (I-T^* T)^{1/2} T^* x, (I - T^* T)^{1/2}T^* x \rangle + \langle TT^*x, (I-TT^*)x \rangle \\
& = & 0. \end{eqnarray*}

Here, to obtain the penultimate equality, we have used the identities $D_\ast^* D_\ast = I - TT^*$ and $C^* D_\ast = - DT^*$. The last equality is a result of elementary formal manipulations. Thus, we get $\alpha_1 (x) = D_\ast x$. Next, we observe :
\begin{eqnarray*} \lefteqn{\langle x - \widehat{T}^{n+1} ( \alpha_n (T^* x) \otimes \mathbf{1}),  \widehat{T}^{n+1} ( \alpha_n (T^*x ) \otimes \mathbf{1}) \rangle} \phantom{zxcvbnmmnbvcxz}\\ & = & \langle \widehat{T}^*(x) - \widehat{T}^n (\alpha_n (T^* x) \otimes \mathbf{1}), \widehat{T}^n (\alpha_n (T^* x) \otimes \mathbf{1}) \rangle \\ & = & \langle T^*(x) - \widehat{T}^n (\alpha_n (T^* x) \otimes \mathbf{1}), \widehat{T}^n (\alpha_n (T^* x) \otimes \mathbf{1}) \rangle \\ & = & 0. \end{eqnarray*}
 Here, the first equality is because of unitarity of $\widehat{T}$. The second equality holds since, for $x \in \mathcal{H}$, we have 
 $\widehat{T}^*(x) - T^* (x) = D_\ast x \otimes \mathbf{1} \in \mathcal{D}_\ast \otimes \mathbf{1}  \perp \widehat{T}^n (\mathcal{D}_\ast \otimes \mathbf{1})$. The last equality is from the definition of $ \alpha_n (\cdot)$. Thus we get $\alpha_{n+1} (x) = \alpha_n (T^*(x))$.  
\end{proof}

Now we are ready to obtain the formula for the characteristic operator :

\begin{theorem} \label{charform} When its domain and co-domain are viewed as Hilbert spaces of vector-valued analytic functions, the characteristic operator $\Theta : \mathcal{D} \otimes H^2 \to \mathcal{D}_\ast \otimes H^2$ of a contraction $T$ is given by $\Theta (f) = (z \mapsto \theta(z)f(z))$, where $\theta : \D \to \mathcal{B}(\mathcal{D}, \mathcal{D}_\ast)$ is the analytic function defined by :
$$ \theta (z)D = D_\ast (I-zT^*)^{-1}(zI-T), \: z \in \D. $$
 \end{theorem}

 \begin{proof} A calculation using the explicit form of $\widehat{T}$ from Theorem \ref{unidil} shows that, for $v=Dx \in \mathcal{D}$, $\widehat{T}^*(v \otimes \mathbf{1}) = (D^*v) + (Cv \otimes \mathbf{1})$. Since $D^*v \in \mathcal{H}$ and $(Cv) \otimes \mathbf{1} \in \mathcal{D}_\ast \otimes \mathbf{1} \subseteq \mathcal{F}_\ast$, Lemma \ref{project} implies that 
 
 \begin{eqnarray*} \Theta (Dx \otimes \mathbf{1}) & = & \mathbf{j}_\ast^*( CDx \otimes \mathbf{1} + D^*Dx)\\
 & = & (CDx \otimes \mathbf{1}) + \mathbf{j}_\ast^*( D^*Dx) \\ 
 & = & -D_\ast Tx \otimes 1 + \sum_{n=1}^\infty \widehat{T}^n ( D_\ast {T^*}^{n-1}(I-T^*T)x \otimes \mathbf{1}). \end{eqnarray*}
 
 Therefore, using $\Psi$ to identify the target with $\mathcal{D}_\ast \otimes H^2$, we get :
 \begin{eqnarray*}  \Theta (Dx \otimes \mathbf{1}) & = & ( z \mapsto -D_\ast Tx +\sum_{n=1}^\infty z^n\,  D_\ast {T^*}^{n-1}(I-T^*T)x \,  \\ & = &  \theta (z) Dx, \end{eqnarray*}
  where $\theta (\cdot)$ is as in the statement of this theorem. Thus, the action of $\Theta$ on the subspace $\mathcal{D} \otimes \mathbf{1}$ of $\mathcal{D}$-valued constant functions is as stated. 
  
  We claim that $\Theta$ intertwines the compressions $\mathbf{j}^*\widehat{T} \mathbf{j} = I \otimes S$ of $\widehat{T}$ on $\mathcal{F} = \mathcal{D} \otimes  H^2$ and $\mathbf{j}_\ast^* \widehat{T} \mathbf{j}_\ast \equiv I \otimes S$ on $ \mathcal{F}_\ast \equiv \mathcal{D}_\ast \otimes H^2$ (after identification via $\Psi$). Granting this claim for the moment, we get, for $v \in \mathcal{D}$ and $m \geq 0$,
   \begin{eqnarray*} \Theta(v \otimes e_m) & = & \Theta ((I \otimes S)^m(v \otimes \mathbf{1})) \\ & = & (I\otimes S)^m (\Theta (v \otimes \mathbf{1})) \\ & = & (z \mapsto e_m(z) \theta (z)v) \\ &  = & (z \mapsto \theta (z)(e_m (z)v)). 
   \end{eqnarray*}     
 
   Thus the action of $\Theta$ on the vectors $v \otimes e_m$ is as stated. Since these vectors span $\mathcal{D} \otimes H^2$, this proves the theorem, subject, of course, to verification of the intertwining property of $\Theta$ claimed above.
   
   To verify this claim, let $\mathbf{p},\; \mathbf{p}_\ast : \widehat{\mathcal{H}} \to \widehat{\mathcal{H}} $ be the orthogonal projections onto the subspaces $\widehat{T}^*(\mathcal{D} \otimes \mathbf{1})$ and $\widehat{T}( \mathcal{D}_\ast \otimes \mathbf{1})$ respectively. Since $\widehat{T}$ is a unitary and $\mathbf{j} \mathbf{j}^*$ and $ \mathbf{j}_\ast \mathbf{j}_\ast^*$ are the orthogonal projections onto $\mathcal{F}$ and $\mathcal{F}_\ast$ respectively, it follows that $\widehat{T}^* (\mathbf{j} \mathbf{j}^*) \widehat{T}$ and $ \widehat{T} (\mathbf{j}_\ast \mathbf{j}_\ast^*) \widehat{T}^*$ are the orthogonal projections onto $\widehat{T}^*(\mathcal{F}) = \mathcal{F} \oplus \widehat{T}^*(\mathcal{D} \otimes \mathbf{1})$ and $ \widehat{T}( \mathcal{F}_\ast) = \mathcal{F}_\ast \ominus \widehat{T}( \mathcal{D}_\ast \otimes \mathbf{1})$, respectively. Therefore, $\widehat{T}^* ( \mathbf{j} \mathbf{j}^*) \widehat{T} = \mathbf{j} \mathbf{j}^* + \mathbf{p}$
   and $ \widehat{T} (\mathbf{j}_\ast \mathbf{j}_\ast^*)\widehat{T}^* = \mathbf{j}_\ast \mathbf{j}_\ast^* - \mathbf{p}_\ast$.
   Also note that $\mathcal{F}$ is orthogonal to both $\widehat{T}^* ( \mathcal{D} \otimes \mathbf{1})$ and $ \widehat{T}(\mathcal{D}_\ast \otimes \mathbf{1})$. Therefore, we get $\mathbf{p} \mathbf{j} = 0 =\mathbf{p}_\ast \mathbf{j}$. Hence,
   $$ (\mathbf{j}_\ast^* \widehat{T}^* \mathbf{j})( \mathbf{j}^* \widehat{T} \mathbf{j}) = \mathbf{j}_\ast^* ( \mathbf{j} \mathbf{j}^* + \mathbf{p}) \mathbf{j} = \mathbf{j}_\ast ^* \mathbf{j},$$
   and,
   $$ (\mathbf{j}_\ast^* \widehat{T} \mathbf{j}_\ast)( \mathbf{j}_\ast^* \widehat{T}^* \mathbf{j}) = \mathbf{j}_\ast^* ( \mathbf{j}_\ast \mathbf{j}_\ast^* - \mathbf{p}_\ast)\mathbf{j}= \mathbf{j}_\ast^* \mathbf{j}. $$
    Thus we get :
   $$ \Theta (I \otimes S) = \mathbf{j}_\ast^* \mathbf{j} = (I \otimes S) \Theta. $$
   This proves the claim.  
   \end{proof}

 The analytic function $\theta$ obtained in this theorem is called the {\em characteristic function} of the contraction $T$. Note that, from its definition, the characteristic operator is clearly a contraction: $ \|\Theta \| \leq 1$. In consequence, the characteristic function is a contraction-valued analytic function: $ \| \theta (z) \| \leq 1 \: \forall z \in \D$.  From its explicit formula, it is easy to verify that $\theta$ is pure contraction valued. While $\theta$ clearly determines $ \Theta$ by the formula $ \Theta(f) = ( z \mapsto \theta (z)f(z))$, specialising this formula, we find that, conversely, $\Theta$ determines $\theta$ by : $\theta(z)v = \Theta (v \otimes \mathbf{1})(z) $. Thus, the characteristic function and the characteristic operator encode the same information about the contraction $T$. 
  
  If $T$ is a contraction and $U$ is a unitary, then $T$ and $T \oplus U$ have the same pair of defect operators and defect spaces. It readily follows that if, as above, $\widehat{T}$ is the minimal unitary dilation of $T$, then the minimal unitary dilation of $T \oplus U$ is $\widehat{T} \oplus U$. In consequence, $T$ and $T \oplus U$ have the same characteristic operator and function. Thus, the characteristic function does not see the unitary parts of the contraction. Therefore, in order that the characteristic function may really characterise the contraction, it is necessary to restrict ourselves to the class of {\em completely non-unitary} (cnu) contractions. Recall that a contraction $T$ is said to be cnu if it has no unitary part (direct summand). Every contraction can be written uniquely as the direct sum of a cnu contraction and a unitary. 
  
 Now, let $T$ and $\tilde{T}$ be two contractions. For each of the constructs attached to $T$ in the above, we shall indicate the corresponding construct for $\tilde{T}$ by a tilde. For instance, $\widetilde{\mathcal{D}}$ is the first defect space of $\tilde{T}$ and $\tilde{\theta}$ is the characteristic function of $\tilde{T}$. We shall say that the characteristic operators $\Theta$ and $\tilde{\Theta}$ {\em coincide} (respectively,the characteristic functions $\theta$ and $\tilde{\theta}$ {\em coincide}) if there are unitaries $v : \mathcal{D} \to \widetilde{\mathcal{D}}$ and $v_* : \mathcal{D}_\ast \to \widetilde{\mathcal{D}}_\ast$ such that $(v_* \otimes I)\Theta = \tilde{\Theta} (v \otimes I)$ (respectively, such that $ v_* \theta(z) = \tilde{\theta}(z) v$ for all $z \in \D$). Clearly, the characteristic operators coincide if and only if the characteristic functions coincide (via the same pair of unitaries).

  Let $\mathcal{G}$ and $\mathcal{G}_\ast$ be the reducing subspaces for 
  $\widehat{T}$ generated by $\mathcal{F}$ and $\mathcal{F}_\ast$, respectively. That is, $\mathcal{G} := \bigoplus_{n= - \infty}^\infty \widehat{T}^n (\mathcal{D} \otimes \mathbf{1})$ and $\mathcal{G}_\ast := \bigoplus_{n=-\infty}^\infty \widehat{T}^n (\mathcal{D}_\ast \otimes \mathbf{1})$. We identify $\mathcal{G}$ and $\mathcal{G}_\ast$ with $\mathcal{D} \otimes L^2(\T)$ and $\mathcal{D}_\ast \otimes L^2(\T)$ in the obvious fashion. The inclusion maps $j: \mathcal{F} \to \hat{\mathcal H}$ and $j_\ast: \mathcal{F}_\ast \to \hat{\mathcal H}$ extend naturally to $\mathcal G$ and $\mathcal G_\ast.$ Since the formula for $j_\ast^\ast(x),$ $x\in \mathcal H,$ obtained in  Lemma \ref{project} remains valid for the extended $j_\ast,$ it follows that the characteristic operator $\Theta$ also extends to $\mathcal G.$ Moreover, it is given by the same formula, namely, $\Theta= j_\ast^\ast \hat{T}j.$ In the proof of the following 
  theorem, $j,$ $j_\ast$ and $\Theta$ denote these extensions.

\begin{theorem} \label{Nagy-Foias}
If $T$ and $\tilde{T}$ are unitarily equivalent contractions then their characteristic functions $\theta$ and $\tilde{\theta}$ coincide. Conversely, if $T$ and $\tilde{T}$ are cnu contractions whose characteristic functions coincide, then $T$ and $\tilde{T}$ are unitarily equivalent. 
\end{theorem}
  
  \begin{proof} First suppose $T$ and $\tilde{T}$ are unitarily equivalent. Say $U : \mathcal{H} \to \widetilde{\mathcal{H}}$ is a unitary such that $UT=\tilde{T}U$. Clearly $U$ restricts to two unitaries $u : \mathcal{D} \to \widetilde{\mathcal{D}}$ and $u_\ast : \mathcal{D}_\ast \to \widetilde{\mathcal{D}}_\ast$. Then it easy to verify that $(u \otimes I) \oplus U \oplus (u_\ast \otimes I)$ intertwines the minimal unitary dilation $\widehat{T}$ and $\widehat{\tilde{T}}$. (This is an instance where the intertwiner between a pair of contractions lifting an intertwiner between the dilations - guaranteed by Lemma \ref{lifting}, can be made explicit.) In consequence, a computation shows that $(u_* \otimes I)\Theta = \tilde{\Theta} (u \otimes I)$. This proves the easy direct part of the theorem. (One needs to be careful here : in the last two sentences, $ u_\ast \otimes I$ refers to two distinct operators, going between different spaces.) This part may also be proved by a direct appeal to the explicit formula (Theorem \ref{charform}) for the characteristic function.   
  
  For the converse, let $T$ and $\tilde{T}$ be cnu contractions such that their characteristic functions $\theta$ and $\tilde{\theta}$ coincide. 
  Since $\mathcal G$ and $\mathcal G_\ast$ are reducing subspaces for $\widehat{T}$, the subspace $\mathcal{M} := (\mathcal{G} + \mathcal{G}_\ast)^\perp$ is a reducing subspace of $\widehat{T}$ contained in $\mathcal{H}$. Therefore, $\mathcal{M}$ is reducing for $T$ and $T|_\mathcal{M} = \widehat{T}|_\mathcal{M}$ is a unitary part of $T$. Since $T$ is cnu, it follows that $\mathcal{M}=\{0\}$. Thus $\overline{\mathcal{G} + \mathcal{G}_\ast} = \widehat{\mathcal{H}}$. Likewise, defining the spaces $\widetilde{\mathcal{G}}$ and $\widetilde{\mathcal{G}}_\ast$ corresponding to the cnu contraction $\tilde{T}$, we get $\overline{\widetilde{\mathcal{G}} + \widetilde{\mathcal{G}}_\ast}= \widehat{\widetilde{\mathcal{H}}}$.
 
 
 Let $v$ and $v_\ast$ be as in the definition of coincidence. 
 With the identification of $\mathcal G$ and $\mathcal G_\ast$ with $\mathcal{D} \otimes L^2(\T)$ and $\mathcal{D}_\ast \otimes L^2(\T)$, respectively, we have the unitaries $v \otimes I : \mathcal{G} \to \widetilde{\mathcal{G}}$ and $v_* \otimes I : \mathcal{G}_\ast \to \widetilde{\mathcal{G}}_\ast$. In view of the preceding paragraph, there is at most one isometry $U : \widehat{\mathcal{H}} \to \widehat{{\widetilde{\mathcal{H}}}}$ which restricts to $(v \otimes I)|_\mathcal{G}$ and $(v_* \otimes I)|_{\mathcal{G}_\ast}$ on $\mathcal{G}$ and $\mathcal{G}_\ast$ respectively. Further, if it exists, then this isometry $U$ is automatically a unitary. We proceed to verify the obvious consistency requirement for the existence of such a unitary extension. For $x \in \mathcal{G}$ and $x_\ast \in \mathcal{G}_\ast$, we have, 
  \begin{eqnarray*} \langle (v \otimes I)x, (v_* \otimes I)x_\ast \rangle & = & \langle \widehat{\tilde{T}}^* (v \otimes I)\hat{T}x, (v_* \otimes I)x_\ast \rangle\\ 
  &=& \langle \widehat{\tilde{T}}^* (v \otimes I)\hat{T}x, \tilde{\mathbf{j}}_\ast (v_* \otimes I)x_\ast \rangle\\
  &=& \langle \tilde{\mathbf{j}}_\ast^* \widehat{\tilde{T}}^* (v \otimes I)\hat{T}x,  (v_* \otimes I)x_\ast \rangle\\
  &=& \langle  (v_\ast \otimes I) \mathbf{j}_\ast^* x,  (v_\ast \otimes I)x_\ast \rangle\\
  &=& \langle \mathbf{j}_\ast^* x,  x_\ast \rangle\\
  &=& \langle  x,  \mathbf{j}_\ast x_\ast \rangle\\
  &=& \langle x, x_\ast \rangle
  \end{eqnarray*}
Here, we have used $\mathbf{j}$ and $\mathbf{j}_\ast$ to denote the inclusion maps from $\mathcal{G}$ and $ \mathcal{G}_\ast$ into $\widehat{\mathcal{H}}$, and likewise for $\tilde{T}$ (we are running out of notations!). The first equality is obtained because of the intertwining relation $(v_* \otimes I)\widehat{T}^* = \widehat{\tilde{T}}^* (v \otimes I).$ The fourth equality is obtained by applying the intertwining relation $ (v_\ast \otimes I) \Theta = \tilde{\Theta} (v \otimes I)$, with the definitions of $\Theta$ and $\tilde{\Theta}$ substituted.  Thus, for $x \in \mathcal{G}$ and $x_\ast \in \mathcal{G}_\ast$, we have $\langle (v \otimes I) x, (v_\ast \otimes I)x_\ast \rangle = \langle x, x_\ast \rangle$. Hence we have the unitary $U : \widehat{\mathcal{H}} \to \widehat{\widetilde{\mathcal{H}}}$ defined by
  $$ U(x + x_\ast) = (v \otimes I)x + (v_* \otimes I)x_\ast $$
  for $x \in \mathcal{G}$ and $x_\ast \in \mathcal{G}_\ast$. Now, it is easy to verify that this unitary intertwines $\widehat{T}$ with $\widehat{\tilde{T}}$ and maps $\mathcal{H}$ onto $\widetilde{\mathcal{H}}$. Hence its restriction to $\mathcal{H}$ is a unitary intertwining $T$ with $\tilde{T}$. Thus $T$ and $\tilde{T}$ are unitarily equivalent.  
  \end{proof}  
    
\section{M\"{o}bius-equivariance of Sz.-Nagy--Foias Theory}
We begin this section by listing some notations to be used throughout the rest of the paper.
  
\begin{notation} \begin{enumerate} \label{notations}
\item[(a)] Choose and fix a Borel square root function $s: \mathbb T \to \mathbb T$, satisfying $s(1) =1$. That is, for each $\beta \in \mathbb T$, $s(\beta) $ is one of the two square roots of $\beta$.   Define the function $c : \mbox{\Mob} \times \D \to \C$  as follows.  For $\varphi$ in M\"{o}b , $\varphi$ can be written uniquely as $\varphi(z) = \beta \tfrac{z-\alpha}{1-\bar{\alpha}z}$, $z\in \mathbb D$, where $\alpha \in \mathbb D$ and $\beta \in \mathbb T$. Then
$$c(\varphi, z) = s(\beta) \frac{\sqrt{1-|\alpha|^2}}{1-\bar{\alpha}z},\,\,z\in\mathbb D.$$
(Thus, $c$ is a function, fixed throughout this paper, which is Borel in the first argument and analytic in the second argument, such that its point-wise square is $(\varphi,z) \mapsto \varphi^\prime (z)$.) 
Notice that for fixed $\varphi \in \mbox{\Mob}$, $z \mapsto  c(\varphi, z)$ is a non-vanishing analytic function on a neighbourhood of $\overline{\D}$. In consequence, for any contraction $T$, the operator $c(\varphi,T)$ (obtained by plugging $T$ into the second slot of $c$) is a well-defined and invertible bounded linear operator. 
\item[(b)] For $\varphi \in \mbox{M\"{o}b},$ let $\varphi^*\in \mbox{M\"{o}b}$ be defined by $\varphi^*(z) = \overline{\varphi(\bar{z}})$, $z\in \mathbb D$. Thus $\varphi \mapsto \varphi^*$ is the unique outer automorphism of \mbox{M\"{o}b} (modulo inner automorphisms). For any projective unitary representation $\sigma$ of \mbox{M\"{o}b}, let $\sigma^\sharp$ denote the representation given by $\sigma^\sharp(\varphi) = \sigma(\varphi^*)$, $\varphi \in \mbox{M\"{o}b}$. Note that if $m$ is the multiplier of $\sigma$, then the multiplier $m^\#$ of $\sigma^\sharp$ is given by the formula 
$$m^\#(\varphi_1, \varphi_2) := m(\varphi_1^*, \varphi_2^*), \,\varphi_1,\varphi_2 \in \mbox{\rm M\"{o}b}.$$  
 \item[(c)] The holomorphic discrete series representation $D_1^+ : \mbox{\Mob} \rightarrow \mathcal{U}(H^2)$ is defined by :
 $$ D_1^+(\varphi^{-1}) f := c(\varphi, \cdot) \cdot (f \circ \varphi), \;\; f \in H^2,\;\; \varphi \in \mbox{\Mob}. $$
 The anti-holomorphic discrete series representation $D_1^- : \mbox{\Mob} \rightarrow \mathcal{U}(H^2)$ is defined by 
$$D_1^-(\varphi) := m_0(\varphi,\varphi^{-1}) D_1^+(\varphi^*).$$ 
Here  $m_0$ denotes the  multiplier of $D_1^+$. 
\end{enumerate}
\end{notation}
\begin{remark}
\begin{enumerate} 
\item[(a)]  Evaluating both sides of the equation 
$D_1^+(\varphi_1\varphi_2) = m_0(\varphi_1,\varphi_2) D_1^+(\varphi_1) D_1^+(\varphi_2)$ at the constant function $\mathbf 1$, we get 
\begin{equation}\label{pm 1}
m_0(\varphi_1, \varphi_2) = \frac{c(\varphi_2^{-1}\varphi_1^{-1}, z)}{c(\varphi_2^{-1},\varphi_1^{-1}(z))c(\varphi_1^{-1},z)},\,\, \varphi_1, \varphi_2 \in \mbox{\rm M\"{o}b},\,\, z\in \mathbb D.
\end{equation}
By the chain rule for differentiation, the square of the right hand side in this equation is equal to $1$. Hence $m_0$ is $\pm 1$ valued. 
We shall often use this observation in what follows, without further mention.  
\item[(b)] Let $\varphi_1, \varphi_2 \in \mbox{\rm M\"{o}b}$ be given by $\varphi_i(z) = \beta_i\tfrac{ z - \alpha_i}{1-\bar{\alpha}_i z}$, $i=1,2$. Let $\varphi:= \varphi_1 \varphi_2$ be given by $\varphi(z) = \beta \tfrac{z - \alpha}{1-\bar{\alpha}z}$. Thus $\alpha_1, \alpha_2 \in \mathbb D$ and $\beta_1, \beta_2\in \mathbb T$.  Also, $\alpha, \beta$ are explicit functions of $\alpha_i$ and $\beta_i$, $i=1,2$.  Then, using the defining formula for $c$ and the Equation \eqref{pm 1}, we get 
\begin{equation} \label{mformula}
m_0(\varphi_1, \varphi_2) = \frac{s(\bar{\beta})}{s(\bar{\beta}_1) s(\bar{\beta}_2)} \,\, \frac{1+\alpha_1 \bar{\alpha}_2 \bar{\beta}_2}{\big|1+\alpha_1 \bar{\alpha}_2 \bar{\beta}_2 \big |}.
\end{equation}
\item[(c)] Specializing Equation \eqref{mformula}, we see 
that if $\varphi \in  \mbox{\rm M\"{o}b}$ is given by $\varphi(z) = \beta \tfrac{z - \alpha}{1-\bar{\alpha}z}$, then we have 
\begin{equation}\label{mdiag}
m_0(\varphi, \varphi^{-1}) = s(\beta) s(\bar{\beta}) = m_0(\varphi^*, {\varphi^*}^{-1}).
\end{equation}
\item[(d)]  Using Equaion \eqref{mformula}, it is easy to verify that $m_0(\varphi_2^{-1}, \varphi_1^{-1})=m_0(\varphi_1^*, \varphi_2^*)$ for $\varphi_1, \varphi_2$ in M\"{o}b.  This equation, together with Equation \eqref{simple} below shows that the representations $D_1^-$   and $D_1^+$ have the common multiplier $m_0$.  
\end{enumerate}
\end{remark}

In the following, we fix a contraction $T$ and a M\"{o}bius map $\varphi$. For each of the constructs corresponding to the (arbitrary but fixed) contraction $T$ introduced above, the corresponding construct for $\varphi (T),$ which is also a contraction, will be indicated by a $\varphi$ in the superscript. For instance, $D^\varphi$ and $D_\ast^\varphi$ are the defect operators for $\varphi (T)$, and so on. The proof of the following Lemma is a straightforward verification and is omitted.  

\begin{lemma} \label{defecttransf} For any contraction $T$ and M\"{o}bius map $\varphi$, we have the identity 
$$ (D^\varphi)^*(D^\varphi) = (D c(\varphi,T))^* (D c(\varphi, T)), \: (D^\varphi_\ast)^*(D^\varphi_\ast) = (D_\ast c(\varphi, T)^*)^*(D_\ast c(\varphi,T)^*). $$ In consequence, there are unitaries (obviously depending on $\varphi$)  $u : \mathcal{D} \to \mathcal{D}^\varphi$ and $u_\ast : \mathcal{D}_\ast \to \mathcal{D}^\varphi_\ast$ which are uniquely determined by the identity 

$$ D^\varphi = uDc(\varphi,T), \: D^\varphi_\ast = u_\ast D_\ast c(\varphi,T)^    *. $$  
\end{lemma}

\begin{lemma}\label{phiofThat}
Let $\widehat{T}$ be the minimal unitary dilation of a contraction $T$. Then for any $\varphi$ in M\"{o}b, $\varphi(\widehat{T})$ has the $3\times 3$ block decomposition 
$\varphi(\widehat{T}) = \big (\!\!\big ( A_{i j}^\varphi\big )\!\!\big)_{1\leq i, j \leq 3}$, where $A_{i j}^\varphi = 0$ for $i > j$ and 
\begin{align*}
A_{11}^\varphi = I\otimes \varphi(S),\, A_{22}^\varphi = \varphi(T),\, & A_{33}^\varphi = I \otimes \varphi(S^*),\\
A_{12}^\varphi =  (I \otimes c( \varphi, S)) \mathbf{i}D c (\varphi, T),\, & A_{23}^\varphi =  c(\varphi, T) (\mathbf{i}_\ast D_\ast)^* \big ( I \otimes c(\varphi, S^*) \big ).
\end{align*}
\end{lemma}
(The explicit formula for $A_{13}^\varphi$ is irrelevant for our purpose.) 
\begin{proof}
Since $\widehat{T}$ has a $3\times 3$ upper triangular block decomposition (given by Theorem \ref{unidil}), it is obvious that so has $\varphi(\widehat{T})$, and its diagonal blocks are as above. Next, we wish  to find the $(1,2)$th entry $A_{12}^\varphi$ of $\varphi (\widehat{T})$. Take $\varphi (z) = \beta (z - \alpha)(1- \overline{\alpha}z)^{-1}$ ( $ |\beta| =1, |\alpha| < 1$). We have $(I - \overline{\alpha} \widehat{T}) \varphi (\widehat{T}) = \beta (\widehat{T} - \alpha I)$. Equating the $(1,2)$th entries of this matrix equation, we get 
\begin{eqnarray*} (I \otimes (I- \overline{\alpha}S))A_{12}^\varphi & = & \mathbf{i}D ( \beta I + \overline{\alpha} \varphi (T)) \\
            & = & \beta (1 - |\alpha|^2) \mathbf{i}D (I-\overline{\alpha}T)^{-1}.\end{eqnarray*}
Therefore, we have 
\begin{eqnarray*} A_{12}^\varphi & = & \beta (1- |\alpha|^2)( I \otimes (I - \overline{\alpha} S)^{-1} ) \mathbf{i}D (I - \overline{\alpha}T)^{-1} \\
                     & = &  (I \otimes c( \varphi, S)) \mathbf{i}D c (\varphi, T). \end{eqnarray*}
A similar computation shows that the $(2,3)$th entry  $A_{23}^\varphi$ of $\varphi (\widehat{T})$ is $c(\varphi, T) (\mathbf{i}_\ast D_\ast)^* \big ( I \otimes c(\varphi, S^*) \big )$. 
\end{proof}

Now we have :
\begin{theorem} \label{charoptransf} For any contraction $T$ and M\"{o}bius map $\varphi$, the minimal unitary dilations $\widehat{\varphi (T)}$ and $\widehat{T}$ (of $\varphi (T)$ and $T$) are related by the formula $ \widehat{ \varphi (T)} V =
V \varphi (\widehat{T})$, where the unitary operator $V : \widehat{\mathcal{H}} \to \widehat{\mathcal{H}}^\varphi $, depending on $\varphi$, is given by 
$$ V:=(m_0(\varphi, \varphi^{-1}) u \otimes D_1^+ (\varphi)) \oplus I \oplus ( m_0(\varphi, \varphi^{-1}) u_\ast \otimes D_1^- (\varphi)), $$
where $u$ and $u_\ast$ are the unitaries given in Lemma \ref{defecttransf}.
 \end{theorem}

\begin{proof}
The identities in Lemma \ref{defecttransf} clearly imply that a contraction $K$ is an isometry/co-isometry/ unitary if and only if $ \varphi (K)$ is. In consequence, $\varphi (\widehat{T})$ is a minimal unitary dilation of $\varphi (T)$. Therefore, we may apply Lemma \ref{lifting} with $T_1 = \varphi (T) = T_2$ and $U=I$ to get a unique unitary $V : \widehat{\mathcal{H}} \to \widehat{\mathcal{H}}^\varphi$ such that $\widehat{\varphi (T)}V = V \varphi (\widehat{T})$ and $V |_{\mathcal{H}} = I$. Now, $(\mathcal{D} \otimes H^2) \oplus \mathcal{H}$ is the unique subspace of $\widehat{\mathcal{H}}$ on which $\varphi (\widehat{T})$ restricts to a minimal isometric dilation of $\varphi (T)$ Similarly, $(\mathcal{D}^\varphi \otimes H^2) \oplus \mathcal{H}$ is the only subspace of $\widehat{\mathcal{H}}^\varphi$ on which $\widehat{\varphi (T)}$ restricts to a minimal isometric dilation of $\varphi (T)$. Since the unitary $V$ intertwines $\varphi (\widehat{T})$ with $\widehat{\varphi(T)}$, it follows that $V$ maps $(\mathcal{D} \otimes H^2) \oplus \mathcal{H}$ onto $(\mathcal{D}^\varphi \otimes H^2) \oplus \mathcal{H}$. Similarly, $V$ maps  $\mathcal{H} \oplus (\mathcal{D}_\ast \otimes H^2)  $ onto $\mathcal{H} \oplus (\mathcal{D}^\varphi_\ast \otimes H^2)$. Since $V$ is a unitary, it follows that $V$ maps the spaces $\mathcal{D} \otimes H^2$, $\mathcal{H}$ and $ \mathcal{D}_\ast \otimes H^2$ onto the corresponding spaces $\mathcal{D}^\varphi \otimes H^2$, $\mathcal{H}$ and $\mathcal{D}_\ast^\varphi \otimes H^2$. In other words, $V$ is a direct sum, say $V = W \oplus I \oplus W_\ast$, where $W : \mathcal{D} \otimes H^2 \to \mathcal{D}^\varphi \otimes H^2$ and $W_\ast: \mathcal{D}_\ast \otimes H^2 \to \mathcal{D}_\ast^\varphi \otimes H^2$ are unitaries. In the block matrix notation, $V$ is a block diagonal : $ V = \rm{diag}(W, I, W_\ast)$. Note that, also, $\widehat{\varphi(T)}$ has the $3\times 3$ upper triangular form given by Theorem \ref{unidil} (with $\varphi(T)$ in place of $T$)  and $\varphi(\widehat{T})$ has the $3\times 3$ upper triangular form given by Lemma \ref{phiofThat}.  Thus, the intertwining relation $ \widehat{ \varphi (T)} V = V \varphi (\widehat{T})$ may be viewed as an equation involving $3 \times 3$ matrices.  Therefore, equating the $(1,1)$th entries in this intertwining relation, we get 
$$(I \otimes S)W = W \varphi (I \otimes S).$$ 
Since $D_1^+$ is the representation associated with $S$, we also have 
$S D_1^+ (\varphi) = D_1^+ (\varphi) \varphi (S)$ and hence 
$$(I \otimes S)(I \otimes D_1^+ (\varphi)) = (I \otimes D_1^+ (\varphi))(\varphi (I \otimes S)).$$ 
Therefore we deduce that 
the unitary $W(I \otimes D_1^+(\varphi))^* : \mathcal{D} \otimes H^2 \to \mathcal{D}^\varphi \otimes  H^2$ intertwines 
$I \otimes S$ on $ \mathcal{D} \otimes H^2$ with $I \otimes S$ on $ \mathcal{D}^\varphi \otimes H^2$. Note that Lemma \ref{defecttransf} implies that $ \mathcal{D}$ and $\mathcal{D}^\varphi$ are isomorphic Hilbert spaces, so that these two avatars of $I \otimes S$ may be identified. Now, the commutant of $I \otimes S$ is well-known. If the Hilbert space on which $I \otimes S$ lives is identified with a Hilbert space of vector-valued functions, then this commutant consists of (multiplication by) operator-valued bounded analytic functions on the disc. In particular, any unitary commuting with $I \otimes S$ must be given by a unitary-valued analytic function. But, by the strong maximum modulus principle, any unitary-valued analytic function is a constant function. Thus, reverting to the tensor product notation, we see that any unitary commuting with $I \otimes S$ must be of the form $w \otimes I$. Coming back to our particular situation, we conclude that there is a unitary $w : \mathcal{D} \to \mathcal{D}^\varphi$ such that $W(I \otimes D_1^+(\varphi))^* = w \otimes I$. That is, $W = w \otimes D_1^+ (\varphi)$. Similarly, since $D_1^-$ is the representation associated with $S^*$, comparing the $(3,3)$th entry of the intertwining relation, it follows that there is a unitary $w_\ast : \mathcal{D}_\ast \to \mathcal{D}_\ast^\varphi$ such that $W_\ast = w_\ast \otimes D_1^- (\varphi)$. Thus, we have $V = (w \otimes D_1^+ (\varphi)) \oplus I \oplus (w_\ast \otimes D_1^- (\varphi))$. To conclude the proof, it now suffices to show that $w=m_0(\varphi, \varphi^{-1}) u$ and $w_\ast = m_0(\varphi, \varphi^{-1}) u_\ast$. 

Now, equating the (1,2)th entries of the intertwining relation (using the new-found diagonal formula for $V$), we get 
$\mathbf{i}^\varphi D^\varphi = (w \otimes D_1^+(\varphi)c(\varphi, S))\mathbf{i}D c(\varphi, T)$. Evaluating both sides at an arbitrary $x \in \mathcal{H}$, we obtain $D^\varphi (x) \otimes \mathbf{1} = (wD c(\varphi, T)x) \otimes (D_1^+(\varphi)c(\varphi, S)\mathbf{1})$. But a little computation shows that $D_1^+ (\varphi) c(\varphi,S) \mathbf{1} = c(\varphi^{-1}, \cdot)c(\varphi, \varphi^{-1}(\cdot))= m_0(\varphi, \varphi^{-1}) \mathbf{1}$ (by Equation \eqref{pm 1} with $\varphi_1 = \varphi$ and $\varphi_2 = \varphi^{-1}$).  So we get $D^\varphi(x) \otimes \mathbf{1} = (m_0(\varphi,\varphi^{-1}) wD c(\varphi, T)x) \otimes \mathbf{1}$ for all $x \in \mathcal{H}$. Hence,  we have $wDc(\varphi,T) = m_0(\varphi, \varphi^{-1}) D^\varphi = m_0(\varphi,\varphi^{-1}) uD c(\varphi,T)$,   where the last equality comes from the defining equation for $u$ from Lemma \ref{defecttransf}. Since $c(\varphi,T)$ is invertible and $D$ has dense range, this forces $w=m_0(\varphi,\varphi^{-1}) u$. 

Similarly, equating the $(2,3)$th entry in the intertwining relation, we obtain 
$(i_*^\varphi D_*^\varphi)^* (w_*\otimes D_1^-(\varphi)) = c(\varphi, T)(i_*D_*)^* (I \otimes c(\varphi, S^*))$.  Evaluating both sides at $x_*\otimes \mathbf 1$, $x_* \in \mathcal D_*$, we get 
\begin{align*} (D_1^-(\varphi) \mathbf 1) (0) (D_*^\varphi)^* w_* x_* &= (c(\varphi,S^*)\mathbf 1)(0) c(\varphi,T)D_*^* x_* \\
&= (c(\varphi,S^*) \mathbf 1)(0) (D_*^\varphi)^* u_* x_*.
\end{align*}
Here the last equality follows from the defining equation for $u_*$ given in Lemma \ref{defecttransf}.  But $(D_1^-(\varphi)\mathbf 1)(0) = m_0(\varphi, \varphi^{-1}) c({\varphi^*}^{-1}, 0)$ and, since $\mathbf 1 \in \ker S^*$, 
$c(\varphi,S^*)(\mathbf 1)(0)=c(\varphi, 0)$. It is easy to verify that
$c({\varphi^*}^{-1}, 0) =  c(\varphi, 0) \neq 0$. Therefore we get 
$(D_*^\varphi)^* w_* x_* = m_0(\varphi, \varphi^{-1}) (D_*^\varphi)^*u_* x_*$ for all $x_*\in \mathcal D_*$. Thus $(D_*^\varphi)^* w_*= m_0(\varphi, \varphi^{-1}) (D_*^\varphi)^*u_*$. Since $(D_*^\varphi)^*$ has trivial kernel, it follows that $w_* = m_0(\varphi, \varphi^{-1}) u_*$. 
\end{proof}

\begin{theorem} \label{covar} Let $\Theta$ and $\theta$ be the characteristic operator and characteristic function of a contraction $T$. Let $\varphi \in \mbox{\Mob}$. Let $\Theta^\varphi$ and $\theta^\varphi$ be the characteristic operator and characteristic function of the contraction $\varphi (T)$. Then we have : 
 
(a) $\Theta^\varphi =(u_\ast \otimes D_1^+ (\varphi)) \Theta (u \otimes D_1^+ (\varphi))^{-1} $. Hence,  $\Theta^\varphi$ coincides with $(I \otimes D_1^+(\varphi))\Theta(I \otimes D_1^+(\varphi))^{-1}$;

 (b) $\theta^\varphi (z) = u_\ast (\theta \circ\varphi^{-1})(z))u^*$
 for all $z \in \D$. Hence,  $\theta^\varphi$ coincides with $\theta \circ \varphi^{-1}$.  
 
 \end{theorem} 
(Here $u$ and $u_\ast$ are the unitaries given in Lemma \ref{defecttransf}.)
\begin{proof} Recall that  $\Theta$ (respectively $\Theta^\varphi$) is the operator of multiplication by $\theta$ (respectively $\theta^\varphi$). Therefore, assuming (b) ( $\theta^\varphi = u_\ast (\theta \circ \varphi^{-1})u^*$), we get that, for $f \in \mathcal{D} \otimes H^2$ and $\varphi$ in M\"{o}b,
\begin{eqnarray*} \Theta^\varphi (u \otimes D_1^+(\varphi))f 
& = & \Theta^\varphi (z \mapsto u c(\varphi^{-1}, z)f(\varphi^{-1}z))\\ 
& = & (z \mapsto \theta^\varphi (z) u c(\varphi^{-1}, z)f(\varphi^{-1}z))\\
& = & (z \mapsto u_\ast \theta(\varphi^{-1}z) c (\varphi^{-1}, z) f (\varphi^{-1}z)) \\
& = & u_\ast D_1^+(\varphi)(z \mapsto \theta(z) f(z)) \\
& = & (u_\ast \otimes D_1^+(\varphi)) \Theta (f), \end{eqnarray*}
so that $\Theta^\varphi(u \otimes D_1^+(\varphi))= (u_\ast \otimes D_1^+(\varphi))\Theta$. Thus, (b) implies (a). Reversing this computation, we see that (a) implies (b). So (a) and (b) are equivalent statements. So it suffices to prove (b).

 To  prove (b), notice the easily verifiable (and wellknown) identities 
 \begin{eqnarray*} c(\varphi, z) c (\varphi,w) & = & \frac{\varphi(z) - \varphi (w)}{z-w} \\
c(\varphi,z)\overline{c(\varphi,w)} & = & \frac{1-\varphi (z) \overline{\varphi (w)}}{1-z \bar{w}}. \end{eqnarray*}
Now, in view of the formula for $\theta$ (and the corresponding formula for $\theta^\varphi$) from Theorem \ref{charform}, and the identities from Lemma \ref{defecttransf} defining the unitaries $u$ and $u_\ast$, we have: 
 \begin{eqnarray*} u_\ast \theta(z) u^* D^\varphi & = & u_\ast \theta (z)D c(\varphi,T) \\
 & = & u_\ast D_\ast (I- zT^*)^{-1}(zI -T) c(\varphi,T) \\
 & = & D_\ast^\varphi {c(\varphi,T)^*}^{-1}(I- zT^*)^{-1}(zI-T)c(\varphi,T) \\
 & = & D_\ast^\varphi (I-\varphi(z) \varphi(T)^*)^{-1}(\varphi(z)I - \varphi(T))\\
 & = & \theta^\varphi (\varphi(z)) D^\varphi. \end{eqnarray*}
 (Here, for the penultimate equality, we have used the two identities displayed above, with $T$ in place of $w$.) Since $D^\varphi$ has dense range, we get $u_\ast \theta(z)u^* = \theta^\varphi (\varphi(z))$. Replacing $\varphi(z)$ by $z$, we obtain (b).
\end{proof}

Remark : While the formula in (b) above (as well as its proof) is well known and essentially already contained in Sz-Nagy--Foias \cite{Na-Fo}, the equivalent formula (a), describing the transformation property of the characteristic operator under the M\"{o}bius group,
is new. So is the obviously closely related Theorem \ref{charoptransf} describing the transformation property of the minimal unitary dilation itself under this group. We believe that, although the characteristic function is a computationally useful tool (as we hope to display further in the planned sequel to this paper), the characteristic operator is theoretically more basic. Accordingly, there ought to be a straightforward (and more revealing) proof of part (a) of Theorem \ref{covar}  directly from the dilation theory (perhaps from Theorem \ref{charoptransf}) without having to go through the formula in part (b), as we have been forced to go. Unfortunately we have failed to find this direct proof. An obstacle is that the characteristic operator does not behave nicely with respect to the intertwining unitary $V$ of Theorem \ref{charoptransf}. In particular, $V$ does not take the range of $\Theta$ to that of $\Theta^\varphi$.

\section{Dilating associated representations and the product formula}

\subsection{Dilation of representations}

Recall from Section 1 that an associator is a Hilbert space operator with an associated (projective) unitary representation of the group M\"{o}b. Thus all associators are homogeneous operators, and all irreducible homogeneous operators are associators. The following theorem is contained in \cite{CM}. However, since a projective representation is a Borel function satisfying a transformation property, to complete the proof of this theorem it is necessary to verify that the extension $\hat{\sigma}$ is a Borel function. This subtelity was overlooked in \cite{CM}. 


\begin{theorem} \label{unirepdil} Let $T$ be a contractive associator on the Hilbert space $\mathcal{H}$. Let $\sigma$ be a 
projective representation of \Mob\ associated with  $T$ and let $\widehat{T} : \widehat{\mathcal{H}} \rightarrow \widehat{\mathcal{H}} $
be the minimal unitary dilation of $T$. Then there is a unique projective representation $\hat{\sigma}$ of \Mob\ living on 
$\widehat{\mathcal{H}}$  such that $\hat{\sigma}$ is associated with $\widehat{T}$ and extends $\sigma$ (that is, $\sigma$ occurs as a direct summand of $\hat{\sigma}$). In consequence, $\widehat{T}$ is an associator.
\end{theorem} 

\begin{proof} Fix $\varphi \in \mbox{\Mob}$. Notice that for any unitary $U$ and for $\varphi \in \mbox{\Mob}$, $\varphi (U)$ is a unitary. Therefore,  $\varphi (\widehat{T})$ is the minimal  unitary dilation of $\varphi (T)$. Therefore, applying Lemma \ref{lifting} to $T_1 = \varphi (T)$, $T_2 = T$ and $U = \sigma (\varphi)$,
we get a unique unitary $\hat{\sigma}(\varphi)$ on $\widehat{\mathcal{H}}$ extending $\sigma (\varphi)$ such that 

$$ \hat{\sigma}(\varphi) \varphi (\widehat{T}) = \widehat{T} \hat{\sigma}(\varphi). $$

This defines a function $\hat{\sigma} : \mbox{\Mob} \rightarrow \mathcal{U}(\widehat{\mathcal{H}})$. To complete the proof, it suffices to show that it is indeed a projective  representation of M\"{o}b. 

Firstly, since $\sigma : \mbox{\Mob} \rightarrow \mathcal{U}(\mathcal{H})$ is a Borel map, the set $\{ (\varphi, \widehat{U}) : \widehat{U}|_{\mathcal{H}} = \sigma (\varphi),\: \widehat{U} \varphi (\widehat{T})= \widehat{T} \widehat{U}\}$ is clearly a Borel subset of $\mbox{\Mob} \times \mathcal{U}(\widehat{\mathcal{H}})$. But, because of the strong uniqueness of the map $\hat{\sigma}$ discussed above, this set is simply the graph of $\hat{\sigma}$.  Thus, $\hat{\sigma}$ is a map between two standard Borel spaces with a Borel graph. Therefore $\hat{\sigma}$ is a Borel function (cf. Theorem  4.5.2 in \cite{SMS}) .

Let $m$ be the multiplier of the projective representation $\sigma$.  Now, $\hat{\sigma} (\varphi_1) \hat{\sigma}(\varphi_2)$ extends $\sigma (\varphi_1) \sigma (\varphi_2)$ and intertwines $(\varphi_1 \varphi_2)(\widehat{T})$ and $\widehat{T}$. Therefore, $m(\varphi_1, \varphi_2)\hat{\sigma}(\varphi_1) \hat{\sigma}(\varphi_2)$ extends $m( \varphi_1, \varphi_2) \sigma (\varphi_1) \sigma ( \varphi_2) = \sigma (\varphi_1 \varphi_2)$ and intertwines $(\varphi_1 \varphi_2)(\widehat{T})$ with $\widehat{T}$. Since $\hat{\sigma}(\varphi_1 \varphi_2)$ does the same thing, the uniqueness statement in Lemma \ref{lifting} implies that 
 $$ \hat{ \sigma} (\varphi_1 \varphi_2) = m(\varphi_1, \varphi_2) \hat{\sigma} (\varphi_1) \hat{\sigma} (\varphi_2) $$
 for all $\varphi_1, \; \varphi_2 \in \mbox{\Mob}$. This proves that $\hat{\sigma}$ is a projective representation.
 \end{proof} 
  
 Clearly, a similar argument shows that the minimal isometric dilation $\tilde{T}$ of a contractive associator $T$ is again an  associator, and any given (projective) representation $\sigma$ associated with $T$ extends uniquely to a representation $\tilde{\sigma}$ associated with $\tilde{T}$. But we can do even better.  We find explicit formulae for $\tilde{\sigma}$ and $\hat{\sigma}$ in terms of $\sigma$. This is the content of the next few results. What is even more surprising is that 
 these formulae involve the  natural discrete series  (holomorphic and anti-holomorphic) projective representations $D_1^{\pm}$ 
 of \Mob\  living on $H^2$.  (Recall Notation 3.1.)

 \begin{theorem} \label{isorepdil} Let $T$ be a contractive associator on a Hilbert space $\mathcal{H}$ with associated projective representation $\sigma$. Let $\tilde{T}$ be the minimal isometric dilation of $T$ on the Hilbert space $(\mathcal{D} \otimes H^2) \oplus \mathcal{H}$ as given by Formula \eqref{isodilform}. Then the unique projective representation $\tilde{\sigma}$ associated with $\tilde{T}$ and having $\sigma$ as a direct summand is given by  $\tilde{\sigma} =(\pi \otimes D_1^+) \oplus \sigma$, where $\pi : \mbox{\Mob} \rightarrow \mathcal{U}(\mathcal{D})$ is the projective representation of \Mob\ given by the formula 
 $$ \pi(\varphi)D = m_0(\varphi, \varphi^{-1}) D \sigma (\varphi) c(\varphi,T)^{-1}, \;\; \varphi \in \mbox{\Mob}\!\!.$$

\end{theorem} 

\begin{proof} Using Equation \eqref{isodilform} and arguing as in the proof of Lemma \ref{phiofThat}, we get 
$$\varphi(\tilde{T}) = \begin{pmatrix} \varphi(I\otimes S) & (I\otimes c(\varphi, S)) \mathbf i D c(\varphi, T)\\ 0 & \varphi(T) \end{pmatrix}.$$

Now, let $\sigma_0$ be the projective representation of \Mob\ living in $\mathcal{D} \otimes H^2$ such that $\tilde{\sigma} = \sigma_0 \oplus \sigma$. Thus, $\sigma_0 (\varphi) \oplus \sigma (\varphi)$ intertwines $\tilde{T}$ and $\varphi (\tilde{T})$.
That is, 
\begin{equation} \label{eqmat} \begin{pmatrix} \sigma_0 (\varphi) & 0 \\ 0 & \sigma (\varphi) \end{pmatrix} \varphi (\tilde{T}) = \tilde{T}  \begin{pmatrix} \sigma_0 (\varphi) & 0 \\ 0 & \sigma (\varphi) \end{pmatrix}. \end{equation}

Substituting the formulae for $\tilde{T}$ and $\varphi(\tilde{T})$ and equating the $(1,1)$th  entry in the resulting matrix equation, and arguing as in the proof of Theorem \ref{charoptransf}, we see that
there is a unique unitary $\pi (\varphi)$ on $\mathcal{D}$ such that $\sigma_0 (\varphi) = \pi (\varphi) \otimes D_1^+(\varphi)$, that is, $\tilde{\sigma}(\varphi) =( \pi ( \varphi) \otimes D_1^+ (\varphi)) \oplus \sigma$. This defines a Borel function $\pi : \mbox{\Mob} \to \mathcal{U}(\mathcal{D})$. Since $D_1^+$, $\sigma$ and $\tilde{\sigma}$ are projective representations of \Mob\ and $\tilde{\sigma} = (\pi \otimes D_1^+) \oplus \sigma$, it  follows
that $\pi$ is a projective representation of M\"{o}b.  It remains to determine the formula for $\pi$.

Substituting in Equation \eqref{eqmat} $\pi (\varphi) \otimes D_1^+(\varphi)$ in place of $\sigma_0(\varphi)$ and equating the 
$(1,2)$th entry and noting that $D_1^+(\varphi)c(\varphi, \cdot) = D_1^+(\varphi)D_1^+(\varphi^{-1})\mathbf 1= m_0(\varphi, \varphi^{-1}) \mathbf 1$, we get 
$\pi(\varphi)D =  m_0(\varphi,\varphi^{-1}) D \sigma (\varphi)c(\varphi,T)^{-1}$. Since $D$ has dense range in $\mathcal{D}$, this formula determines $\pi$ uniquely. \end{proof}


\begin{theorem} \label{coisorepdil} If $\sigma$ is  an associated representation of a contractive associator $T$ then the unique projective representation $\tilde{\sigma}_\ast$ extending $\sigma$ and associated with the minimal co-isometric dilation $\tilde{T}_\ast$ of $T$ is given by $\tilde{\sigma}_\ast = \sigma \oplus (\pi_\ast \otimes D_1^-)$, where $\pi_*:\mbox{\rm M\"{o}b} \to \mathcal B(\mathcal D_*)$ is the projective representation of \Mob\ given by the formula 
$$ \pi_\ast (\varphi)D_\ast = m_0(\varphi, \varphi^{-1}) D_\ast \sigma (\varphi) {c(\varphi, T)^{-1}}^*, \: \varphi \in \mbox{\Mob}.$$
\end{theorem}

\begin{proof}
Arguing as in the proof of Theorem \ref{isorepdil}, we see that $\tilde{\sigma}_*=\sigma\oplus (\pi_*\otimes D_1^-)$, where $\pi_*$ is a projective representation of \Mob\ living on $\mathcal D_*$. 
To find the formula for $\pi_*$, note that ${\tilde{T}_*}^*$ is the minimal isometric dilation of the contractive associator $T^*$ and $\sigma^\#$ is an associated representation of $T^*$. Thus ${\tilde{\sigma}_*}^\# = \sigma^\# \oplus (\pi_*^\# \otimes m_0(\varphi, \varphi^{-1}) D_1^+)$ is the unique representation associated with ${\tilde{T}_*}^*$ and 
extending $\sigma^\#$. Therefore, by Theorem \ref{isorepdil} applied to $T^*$, we see that $\pi_*^\#(\varphi) D_* =  D_* \sigma^\#(\varphi) c(\varphi, T^*)^{-1}$. That is, 
$$\pi_*(\varphi) D_* = \pi_*^\#(\varphi^*) D_* =  D_* \sigma(\varphi) c(\varphi^*, T^*)^{-1}.$$ But it is easy to see from Equation \eqref{mdiag} that $c(\varphi^*, \bar{z}) = m_0(\varphi, \varphi^{-1}) \overline{c(\varphi,z)}$.  Hence $c(\varphi^*,T^*) =m_0(\varphi, \varphi^{-1}) c(\varphi,T)^*$. Thus we get the above formula for $\pi_*$.  
\end{proof}

Since the minimal unitary dilation $\widehat{T}$ of a contraction $T$ is built by gluing together its (minimal) isometric and co-isometric dilations as described in Theorem \ref{unidil}, it follows that :

\begin{theorem} \label{unirepdilform} If $T$ is a contractive associator with associated representation $\sigma$ then the unique projective representation $\hat{\sigma}$ extending $\sigma$ and associated with the minimal unitary dilation $\widehat{T}$ of $T$ is given by the formula
$$ \hat{\sigma} = (\pi \otimes D_1^+) \oplus \sigma \oplus (\pi_\ast \otimes D_1^-), $$
where $\pi$ and $\pi_\ast$ are the projective representations of \Mob\ living on $\mathcal{D}$ and $\mathcal{D}_\ast$ given by Theorems \ref{isorepdil} and \ref{coisorepdil}. \end{theorem}


The appearance of the representations $\pi$ and $\pi_\ast$ in these two theorems is rather mysterious. We give below a direct verification that these are indeed projective representations of the group M\"{o}b.

\begin{theorem} \label{tworeps} Let $T$ be a contractive associator with associated representation $\sigma$. Then the functions $\pi : \mbox{\Mob} \to \mathcal{U}(\mathcal{D})$ and $\pi_\ast : \mbox{\Mob} \to \mathcal{U}(\mathcal{D}_\ast)$ given by the formula
$$ \pi (\varphi)D = m_0(\varphi, \varphi^{-1})D \sigma (\varphi) c(\varphi, T)^{-1}, \: \pi_\ast (\varphi)D_\ast = m_0(\varphi, \varphi^{-1}) D_\ast \sigma (\varphi){c(\varphi, T)^{-1}}^*$$
are projective representations of M\"{o}b. If $m$ is the multiplier of $\sigma$ then the common multiplier of $\pi$ and $\pi_\ast$ is $m \cdot m_0$ (pointwise product). 
\end{theorem} 

\begin{proof} Clearly $\pi$ and $\pi_\ast$ are Borel functions.  For $\varphi_1, \varphi_2$ in \Mob\,  we get from Equation \ref{pm 1}
\begin{eqnarray*} \pi(\varphi_1)\pi(\varphi_2) D & = & m_0(\varphi_2, \varphi_2^{-1}) \pi(\varphi_1) D \sigma (\varphi_2) c(\varphi_2, T)^{-1}\\
& = &m_0(\varphi_1, \varphi_1^{-1}) m_0(\varphi_2, \varphi_2^{-1})  D\sigma (\varphi_1) c(\varphi_1, T)^{-1} \sigma (\varphi_2) c(\varphi_2, T)^{-1} \\
& = &m_0(\varphi_1, \varphi_1^{-1}) m_0(\varphi_2, \varphi_2^{-1}) D \sigma(\varphi_1) \sigma (\varphi_2) c(\varphi_1, \varphi_2 (T))^{-1} c (\varphi_2, T)^{-1} \\
& = & \bar{m}(\varphi_1, \varphi_2) m_0(\varphi_1, \varphi_1^{-1}) m_0(\varphi_2, \varphi_2^{-1}) {m}_{0}(\varphi_2^{-1}, \varphi_1^{-1}) D \sigma (\varphi_1 \varphi_2) c(\varphi_1\varphi_2, T)^{-1}\\
& = & \bar{m}(\varphi_1, \varphi_2) m_0(\varphi_1, \varphi_1^{-1}) m_0(\varphi_2, \varphi_2^{-1}) {m}_{0}(\varphi_2^{-1}, \varphi_1^{-1}) m_0(\varphi_1\varphi_2,\varphi_2^{-1}\varphi_1^{-1}) \pi(\varphi_1 \varphi_2) D.
\end{eqnarray*}
Applying Equation \eqref{mult} with $\varphi_3 = \varphi_2^{-1} \varphi_1^{-1}$, we get 
$$m_0(\varphi_1, \varphi_2) m_0(\varphi_1\varphi_2, \varphi_2^{-1}\varphi_1^{-1}) = m_0(\varphi_1,\varphi_1^{-1})
m_0(\varphi_2,\varphi_2^{-1} \varphi_1^{-1}).$$
Applying Equation \eqref{mult} to $m_0$ after the substitutions $\varphi_1 \mapsto \varphi_2,$ $\varphi_2 \mapsto \varphi_2^{-1}$, $\varphi_3 \mapsto \varphi_1^{-1}$, we get 
$$m_0(\varphi_2,\varphi_2^{-1}) = m_0(\varphi_2,\varphi_2^{-1}\varphi_1^{-1}) m_0(\varphi_2^{-1}, \varphi_1^{-1}).$$
Thus
$$m_0(\varphi_1, \varphi_2) m_0(\varphi_1\varphi_2, \varphi_2^{-1}\varphi_1^{-1}) m_0(\varphi_1,\varphi_1^{-1}) = m_0(\varphi_2,\varphi_2^{-1}\varphi_1^{-1}) = m_0(\varphi_2, \varphi_2^{-1}) m_0(\varphi_2^{-1}, \varphi_1^{-1}).$$
Hence we get 
\begin{equation}\label{simple}
m_0(\varphi_1,\varphi_1^{-1})m_0(\varphi_2,\varphi_2^{-1}) m_0(\varphi_1\varphi_2, \varphi_2^{-1}\varphi_1^{-1}) m_0(\varphi_2^{-1}, \varphi_1^{-1}) = m_0(\varphi_1, \varphi_2).
\end{equation}

Thus, $\pi(\varphi_1 \varphi_2) = (m \cdot m_0)(\varphi_1, \varphi_2) \pi(\varphi_1) \pi(\varphi_2)$. It is clear from the formula for $\pi$ that $\pi({\rm id})=I$. Therefore, to show that $\pi$ is a (projective unitary) representation with multiplier $m \cdot m_0$, it only remains to establish that $\pi (\varphi)$ is unitary for all $ \varphi$ in M\"{o}b. In view of the `quasi-homomorphism' property of $\pi$ already verified, it suffices to show that $\pi (\varphi)$ is an isometry. 
%
Now, to verify that $\pi (\varphi)$ is an isometry, one observes:
\begin{eqnarray*} \langle \pi (\varphi)Dx, \pi (\varphi)Dy \rangle 
& = & \langle D \sigma (\varphi) c(\varphi,T)^{-1}x,  D \sigma (\varphi) c(\varphi,T)^{-1} y \rangle \\ 
& = & \langle \sigma (\varphi)^* D^*D \sigma (\varphi)c(\varphi,T)^{-1}x, c(\varphi,T)^{-1}y \rangle \\
& = & \langle {(D^\varphi)}^*D^\varphi c(\varphi, T)^{-1} x, c(\varphi, T)^{-1} y \rangle\\
& = & \langle c(\varphi, T)^* D^*Dx, c(\varphi, T)^{-1} y \rangle \\
& = & \langle D^* D x, y \rangle\\
& = & \langle Dx,Dy \rangle \end{eqnarray*}
for $x,y \in \mathcal{H}$. Here to get the fourth equality, we have used Lemma \ref{defecttransf}. This proves that $\pi$ is indeed a projective unitary representation with multiplier $m \cdot m_0$. An analogous calculation shows that $\pi_*$ is also a projective unitary representation with multiplier $m \cdot m_0$.  
\end{proof}

\subsection{The Product Formula}

In this subsection we present the product formula for homogeneous characteristic functions.  An existential proof of the direct part of this result  was given in \cite{survey} for the class of irreducible homogeneous contractions. 
The theorem  presented here is  constructive and works in the greater generality of contractive associators. (Recall from \cite[Theorem 2.2]{shift} that all irreducible homogeneous operators are associators; but the converse is not true.)  Moreover, the converse part of the following Theorem is entirely new. 

Let $A(\mathbb D)$ be the Banach space of all continuous functions on $\bar{\mathbb D}$ which are holomorphic on $\mathbb D$, with sup norm. Clearly, M\"{o}b  may be viewed as a subset of $A(\mathbb D)$. 

\begin{lemma} \label{total}
M\"{o}b is a total set in $A(\mathbb D)$. 
\end{lemma}
\begin{proof}
It suffices to show that if $\eta$ is a bounded linear functional on $A(\mathbb D)$ such that $\eta(\varphi) = 0$ for all $\varphi\in \mbox{\rm M\"{o}b}$,then $\eta\cong 0$. Note that $A(\mathbb D)$ may be viewed as a closed subspace of $C(\mathbb T)$, the space of continuous functions on $\mathbb T$. Therefore, by Hahn-Banach Theorem and the Riesz representation Theorem, there is a complex Borel measure $\mu$ on $\mathbb T$ such that $\eta(f) = \int f d\mu$, $f\in A(\mathbb D)$. 
For $n \geq  0,$ let $e_n$ in $A(\mathbb D)$ be the function $z\mapsto z^n$. By Mergelyan's Theorem, the set $\{e_n: n\geq 0\}$ is a total set in $A(\mathbb D)$. Therefore, to complete the proof, it suffices to show that $\int \varphi d\mu = 0$ for all $\varphi$ in \Mob implies $\int e_n d\mu = 0$ for all $n \geq 0$. 
Take $\varphi(z) = \frac{z-\alpha}{1-\bar{\alpha} z},$ where $\alpha \in \mathbb D$ is arbitrary. Note that we have the representation (Taylor expansion) 
$$\varphi = -\alpha e_0 + \sum_{n=1}^\infty (1- \alpha \bar{\alpha})\bar{\alpha}^{n-1} e_n,$$
where the series converges in the norm of $A(\mathbb D)$. Therefore, integrating with respect to $\mu$, we get 
$$-\alpha \int e_0 d\mu + \sum_{n=1}^\infty (1-\alpha \bar{\alpha})\bar{\alpha}^{n-1}\int e_n d\mu = 0$$
for all $\alpha \in \mathbb D$. Hence $\int e_n d\mu = 0$ for all $n\geq 0$.  
\end{proof}

\begin{notation} For $z \in \D$, $\varphi_z$ is the unique involution (element of order $2$) in \Mob\ which interchanges $0$
and $z$. Explicitly, we have $\varphi_z(w)=\frac{z-w}{1-\bar{z}w}$ for $w \in \D$. Also, $\K = \{ \varphi \in \mbox{\Mob} :~ \varphi(0) = 0 \} =\{z \mapsto \beta z : \beta \in \mathbb T \}$ is the  
standard maximal compact subgroup of M\"{o}b.
\end{notation}

\begin{theorem} \label{proform}
Let $T$ be a cnu contractive associator with associated representation $\sigma$. Then its characteristic function $\theta$ is given by $\theta (z) = \pi_\ast (\varphi_z)^* C \pi(\varphi_z), \: z \in \D $, where 
$\pi$ and $\pi_\ast$ are the representations (living on $\mathcal{D}$ and $\mathcal{D}_\ast$) given in Theorem \ref{tworeps}, and $C : \mathcal{D} \to \mathcal{D}_\ast $ is the pure contraction given by $Cx =-Tx, \: x \in \mathcal{D}$ (which intertwines  $\pi|_{\K}$ and $\pi_\ast|_{\K}$).

Conversely, if $\pi_\ast,~\pi$ are \pr unitary \reps of \Mob\ with a
common multiplier, and $C$ 
is a purely  contractive intertwiner between $\pi |_{\K}$ and $\pi_\ast
|_{\K}$ such that the function $\theta$ defined by $\theta (z) = \pi_\ast
(\varphi_z)^*C \pi (\varphi_z)$   is analytic on $\D$, then $\theta$ is the
\cf\  of a cnu contractive associator. 
\end{theorem}

\begin{proof}
Using the formulae for $\pi (\varphi)D$ and $\pi_\ast (\varphi)D_\ast$ from Theorem \ref{tworeps} and 
%
the easy identity $CD = -D_\ast T$,  we get, when $T$ is a cnu contractive associator and $\varphi \in \mbox{M\"{o}b}$,  
\begin{eqnarray*} \pi_\ast (\varphi)^* C \pi (\varphi)D & = &m_0(\varphi, \varphi^{-1}) \pi_\ast (\varphi)^* CD \sigma(\varphi) c(\varphi, T)^{-1} \\
 & = & -  m_0(\varphi, \varphi^{-1})\pi_\ast (\varphi)^*D_\ast T \sigma (\varphi) c(\varphi, T)^{-1} \\
 & = &  - D_\ast  {c(\varphi, T)}^* \sigma (\varphi)^* T \sigma (\varphi) c(\varphi, T)^{-1} \\
 & = & -D_\ast {c(\varphi, T)}^* \varphi(T) c(\varphi, T)^{-1}. 
 \end{eqnarray*}
Thus 
\begin{equation} \label{eq:1}
\pi_*(\varphi)^* C\pi(\varphi)D = - D_* c(\varphi,T)^* \varphi(T) c(\varphi,T)^{-1},\, \varphi \in \mbox{M\"{o}b}.
\end{equation}  
Taking $\varphi = k\in \mathbb K$ in Equation \eqref{eq:1}, we get $\pi_*(k)^* C\pi(k) D = - D_*T= CD$, $k\in \mathbb K$. Since $D$ has dense range, this shows that 
$$C\pi(k) = \pi_*(k) C,\,k\in \mathbb K.$$ 
Thus $C$ intertwines $\pi|_\mathbb K$ and $\pi_*|\mathbb K$. 

Taking $\varphi = \varphi_z$, $z\in \mathbb D$,  in Equation \eqref{eq:1} and noting that 
$$c(\varphi_z, T)^* \varphi_z(T) c(\varphi_z, T)^{-1} = (I - zT^*)^{-1} (T-zI),$$ 
we get $\pi_*(\varphi_z)^* C\pi(\varphi_z)D = D_* (I-zT^*)^{-1} (z I - T) = \theta(z) D$ (by Theorem \ref{charform}).    
Since $D$ has dense range, this proves the product formula for $\theta$.                  

For the converse, let $\theta (z) := \pi_\ast (\varphi_z)^*C \pi
(\varphi_z)$ be an analytic function. Since $C$ is a pure
contraction and $\theta (z)$ is obtained from  $C$ by pre- and post-multiplying by unitaries, it follows that  
$\theta$ is pure contraction valued. Hence by \cite{Na-Fo}, $\theta$ is the \cf\ of a
cnu \con $T:\mathcal H \to \mathcal H$. Since $\varphi_0\in \mathbb K$, $C$ intertwines $\pi(\varphi_0)$ with $\pi_*(\varphi_0)$. Therefore $C=\pi_*(\varphi_0)^*C\pi(\varphi_0)= \theta(0)$.  Since $\theta$ is given in terms of $T$ by Theorem  \ref{charform}, it follows that the domain and codomain of $C$ are the defect spaces $\mathcal D$, $\mathcal D_*$ of $T$. For $\varphi \in \mbox{\Mob}$ and $w
\in \D$, write $\varphi_w \varphi = k \varphi_z$ where $k \in \K$ and
$z = (\varphi_w \varphi)^{-1}(0)= \varphi^{-1}(w)$. Then we have
\begin{eqnarray*} 
\pi_\ast (\varphi)^* \theta (w) \pi (\varphi) & = & \pi_\ast (\varphi)^* \pi_\ast
(\varphi_w)^* C \pi (\varphi_w) \pi (\varphi)\\
& = & \pi_\ast (\varphi_w \varphi)^* C \pi(\varphi_w \varphi) \\
& = & \pi_\ast (k \varphi_z)^* C \pi (k\varphi_z) \\
& = & \pi_\ast (\varphi_z)^* \pi_\ast (k)^* C \pi (k) \pi (\varphi_z) \\
& = & \pi_\ast (\varphi_z)^* C \pi (\varphi_z) \\
& = & \theta (\varphi^{-1}(w)). 
\end{eqnarray*}
(Here, for the second and fourth equality we have used the assumption
that $\pi_\ast$ and $\pi$ are \pr  \reps with a common multiplier. For
the penultimate equality, the assumption that $C$ intertwines $\pi
|_{\K}$ and $\pi_\ast |_{\K}$ has been used.) 
Now, part (b) of Theorem \ref{covar} implies that the characteristic function  $\theta^{\varphi}$ of $\varphi(T)$ is related to $\theta$ by the equation 
$$\theta^\varphi(w) = u_*(\varphi)\pi_*(\varphi)^* \theta(w) (u(\varphi) \pi(\varphi)^*)^*,$$
where $u(\varphi)$,  $u_*(\varphi)$ are the unitaries given by Lemma \ref{defecttransf}.  Thus, $\theta^\varphi$ coincides with $\theta$.
Therefore, following the proof of Theorem \ref{Nagy-Foias} (with $\tilde{T} = \varphi(T)$ and $v(\varphi) = u(\varphi) \pi(\varphi)^*, v_*(\varphi) = u_*(\varphi) \pi_*(\varphi)^*$) we get a unitary $U(\varphi):\widehat{\mathcal H} \to \widehat{\mathcal H}^\varphi$ such that $U(\varphi)$ maps $\mathcal H$ onto $\mathcal H$ and $\widehat{\varphi(T)}=U(\varphi) \widehat{T}U(\varphi)^*$. Let $\sigma(\varphi):\mathcal H\to \mathcal H
$ be the unitary obtained by restricting $m_0(\varphi, \varphi^{-1}) U(\varphi)^*$ to $\mathcal H$. Since $T$ and $\varphi(T)$ are the compressions of $\widehat{T}$ and $\widehat{\varphi(T)}$ (respectively) to $\mathcal H$, it follows that $\varphi(T) = \sigma(\varphi)^* T \sigma(\varphi),$ $\varphi\in \mbox{\rm M\"{o}b}$.  Therefore, to complete the proof, it suffices to show that $\varphi \mapsto \sigma(\varphi)$ is a projective unitary representation of M\"{o}b. We shall do so by finding an explicit formula for $\sigma(\varphi)$. Indeed we shall find a total subset $Z$ of $\mathcal H$ on which $\sigma(\varphi)$ acts as a signed permutation. 

Explicitly, for $n\geq 0$, let $e_n\in A(\mathbb D)$ be as in the proof of Lemma \ref{total}. Then the proof of Theorem \ref{Nagy-Foias} (converse part) shows that 
$$X:=\big \{e_n(\widehat{T})(x\otimes \mathbf 1): n\geq 0, x\in \mathcal D\cup \mathcal D_*\} \cup \{e_n(\widehat{T})^* (x\otimes \mathbf 1): n\geq 0, x\in \mathcal D\cup \mathcal D_* \big \}$$
is a total set in $\widehat{\mathcal H}$, and $U(\varphi)$ is given on this total set by the formulae (for $n\geq 0$, $x\in \mathcal D, x_*\in \mathcal D_*$) 
\begin{eqnarray*}
U(\varphi)(e_n(\widehat{T})(x\otimes \mathbf 1) ) &=& e_n (\widehat{\varphi(T)}) (u(\varphi) \pi(\varphi)^* x \otimes \mathbf 1)
\\
U(\varphi)(e_n(\widehat{T})^* (x\otimes \mathbf 1) )&=& e_n (\widehat{\varphi(T)})^* (u(\varphi) \pi(\varphi)^* x \otimes \mathbf 1)\\
U(\varphi)(e_n(\widehat{T})(x_*\otimes \mathbf 1) ) &=& e_n (\widehat{\varphi(T)}) (u_*(\varphi) \pi_*(\varphi)^* x_* \otimes \mathbf 1)\\
U(\varphi)(e_n(\widehat{T})^* (x_*\otimes \mathbf 1) ) &=& e_n (\widehat{\varphi(T)})^*(u_*(\varphi) \pi_*(\varphi)^* x_* \otimes \mathbf 1).
\end{eqnarray*}
Since, by the spectral theorem, $f\mapsto f(\widehat{T})$ is a contractive linear transformation from $A(\mathbb D)$ to $\mathcal B(\hat{\mathcal H})$ and since the set $X$ above is total in $\widehat{\mathcal H}$, Lemma \ref{total} implies that the set 
$$Y:=\{\psi(\widehat{T})(x\otimes \mathbf 1): \psi \in \mbox{\rm M\"{o}b}, x\in \mathcal D\cup \mathcal D_*\} \cup \{\psi(\widehat{T})^* (x\otimes \mathbf 1): \psi \in \mbox{\rm M\"{o}b}, x\in \mathcal D\cup \mathcal D_* \}$$
is also a total set in $\widehat{\mathcal H}$, and $U(\varphi)$ is given on this total set by the formulae (for $\psi \in \mbox{\rm M\"{o}b}$, $x\in\mathcal D$, $x_*\in \mathcal D_*$) 
\begin{eqnarray*}
U(\varphi)(\psi(\widehat{T})(x\otimes \mathbf 1) )&=& \psi (\widehat{\varphi(T)}) (u(\varphi) \pi(\varphi)^* x \otimes \mathbf 1)
\\
U(\varphi)(\psi(\widehat{T})^* (x\otimes \mathbf 1) )&=& \psi (\widehat{\varphi(T)})^* (u(\varphi) \pi(\varphi)^* x \otimes \mathbf 1)\\
U(\varphi)(\psi(\widehat{T})(x_*\otimes \mathbf 1) )&=& \psi (\widehat{\varphi(T)}) (u_*(\varphi) \pi_*(\varphi)^* x_* \otimes \mathbf 1)\\
U(\varphi)(\psi(\widehat{T})^* (x_*\otimes \mathbf 1) )&=& \psi (\widehat{\varphi(T)})^*(u_*(\varphi) \pi_*(\varphi)^* x_* \otimes \mathbf 1).
\end{eqnarray*}
For $\psi$ in M\"{o}b, $x\in \mathcal D$, $x_*\in \mathcal D_*$, define the vectors $v[\psi,x], v_*[\psi,x_*]$ in $\mathcal H$ by 
$$
v[\psi,x]:= c(\psi, T)^* D^* x, \,\,
v_*[\psi,x_*]:= c(\psi, T)D_*^* x_*.
$$
 Let $p$ be the orthogonal projection from $\widehat{\mathcal H}$ onto $\mathcal H$. It follows from Lemma \ref{phiofThat} that, for $x\in \mathcal D$, $x_* \in \mathcal D_*$, 
$$
p(\psi(\widehat{T}) (x \otimes \boldsymbol 1)) = 0 = p(\psi(\widehat{T})^*(x_* \otimes \boldsymbol 1)),
$$
and, in terms of the notation in Lemma \ref{phiofThat}, 
\begin{equation}\label{pofpsi}
\begin{cases}
p(\psi(\widehat{T})^*(x \otimes \boldsymbol 1)) &\!\!\!\!=\big (A_{12}^\psi\big )^* (x\otimes \boldsymbol 1)
= \overline{c(\psi,0)} v[\psi,x],\\
p(\psi(\widehat{T})(x_* \otimes \boldsymbol 1)) &\!\!\!\!= A_{23}^\psi (x_*\otimes \boldsymbol 1)
= c(\psi,0) v_*[\psi, x_*].
\end{cases}
\end{equation}
Since the image under $p$ of the total subset $Y$ of $\widehat{\mathcal H}$ is a total subset of $\mathcal H$, and since $c(\psi, 0) \neq 0$, it follows that the set 
\begin{equation}\label{totalZ}
Z:= \{ v[\psi,x]:\psi \in \mbox{\rm M\"{o}b},\,\, x\in \mathcal D\}  \cup \{ v_*[\psi,x_*]:\psi \in \mbox{\rm M\"{o}b},\,\, x_*\in \mathcal D_*\}  
\end{equation}
is a total subset of $\mathcal H$. 

Let $V(\varphi):\widehat{\mathcal H} \to \widehat{\mathcal H}^\varphi$ be the $3\times 3$ block diagonal unitary given by 
$$V(\varphi):= \mbox{\rm diag} (m_0(\varphi,\varphi^{-1}) u(\varphi) \otimes D_1^+(\varphi),\, I,\,m_0(\varphi, \varphi^{-1}) u_*(\varphi)\otimes D_1^-(\varphi)).$$
By Theorem \ref{charoptransf}, we have $\widehat{\varphi(T)} = V(\varphi) \varphi(\widehat{T}) V(\varphi)^*,$ and hence $\psi(\widehat{\varphi(T)})= 
V(\varphi) (\psi \varphi)(\widehat{T})V(\varphi)^*.$
Therefore, letting $q$ be the orthogonal projection from $\widehat{\mathcal H^\varphi}$ onto $\mathcal H$, we have 
\begin{eqnarray*}
q U(\varphi) (\psi(\widehat{T})^* (x\otimes 1) )&=& q(\psi(\widehat{\varphi(T)})^*(u(\varphi)\pi(\varphi)^* x \otimes \boldsymbol 1))\\ 
&=&m_0(\varphi, \varphi^{-1}) \big (A_{12}^{\psi\varphi}\big )^* (\pi(\varphi)^* x\otimes D_1^+(\varphi)^* \boldsymbol 1)\\ 
&=&m_0(\varphi,\varphi^{-1}) (c(\psi \varphi, S)^* D_1^+(\varphi)^* \mathbf 1) (0) v[\psi \varphi, \pi(\varphi)^* x].\\
q U(\varphi) (\psi(\widehat{T}) (x_*\otimes 1) ) &=& q(\psi(\widehat{\varphi(T)})(u_*(\varphi) \pi_*(\varphi)^* x_* \otimes \boldsymbol 1))\\ 
&=& A_{23}^{\psi \varphi} (\pi_*(\varphi)^* x_*\otimes D_1^-(\varphi)^* \boldsymbol 1)\\
&=&m_0(\varphi, \varphi^{-1}) ( c(\psi \varphi, S^*) D_1^-(\varphi)^*\mathbf 1)(0) v_*[\psi\varphi, \pi_*(\varphi)^*x_*].
\end{eqnarray*}
Now we compute
\begin{align*}
(c(\psi \varphi, S)^* D_1^+(\varphi)^* \mathbf 1) (0) &= \langle c(\psi \varphi, S)^* D_1^+(\varphi)^* \mathbf 1 , \mathbf 1\rangle \\
&=\langle \mathbf 1,  D_1^+(\varphi) c(\psi \varphi, S) \mathbf 1 \rangle\\
&= \langle \mathbf 1, D_1^+(\varphi) c(\psi \varphi, \cdot) \rangle\\
&=\langle \mathbf 1, c(\varphi^{-1}, \cdot) c(\psi \varphi, \varphi^{-1}(\cdot) ) \rangle \\
&= m_0(\varphi, \varphi^{-1}\psi^{-1}) \langle \mathbf 1, c(\psi, \cdot )\rangle\\
&= \overline{c(\psi,0)}m_0(\varphi, \varphi^{-1}\psi^{-1}).
\end{align*}
Here the penultimate equality is by Equation \eqref{pm 1} with $\varphi_1 = \varphi$, $\varphi_2 = \varphi^{-1} \psi^{-1}$.

Note that, up to scaling, $c(\varphi^*, \cdot)$ is the Szeg\"{o} kernel at $\overline{\varphi^{-1}(0)}$. Therefore, $c(\varphi^*, \cdot)$ is an eigenvector of $S^*$ with eigenvalue $\varphi^{-1}(0)$. Hence $c(\varphi^*,\cdot)$ is an eigenvector of $c(\psi\varphi, S^*)$ with eigenvalue $c(\psi\varphi, \varphi^{-1}(0))$.  Hence we get 
\begin{align*}
(c(\psi \varphi, S^*) D_1^-(\varphi)^* \mathbf 1) (0) &= m_0(\varphi, \varphi^{-1}) (c(\psi \varphi, S^*) D_1^- (\varphi^{-1}) \mathbf 1)(0)\\
&= (c(\psi \varphi, S^*) c(\varphi^*, \cdot))(0)\\
&=  c(\psi \varphi, \varphi^{-1}(0)) c(\varphi^*, 0)\\
&= c(\psi \varphi, \varphi^{-1}(0)) c(\varphi^{-1}, 0)\\
&= c(\psi, 0)  m_0(\varphi, \varphi^{-1}\psi^{-1}).
\end{align*}
(Here the last equality is by \eqref{pm 1} with $\varphi_1 = \varphi$, $\varphi_2 = \varphi^{-1}\psi^{-1}$, $z=0$.) Therefore we have 
\begin{equation}\label{qofpsi}
\begin{cases}
q U(\varphi)(\psi(\widehat{T})^* (x\otimes 1) )&=\overline{c(\psi,0)}  m_0(\varphi, \varphi^{-1}) m_0(\varphi, \varphi^{-1}\psi^{-1}) v[\psi \varphi, \pi(\varphi)^* x]\\
q U(\varphi)(\psi(\widehat{T}) (x_*\otimes 1)) &=c(\psi,0)  m_0(\varphi, \varphi^{-1}) m_0(\varphi, \varphi^{-1}\psi^{-1}) v_*[\psi \varphi, \pi_*(\varphi)^* x_*].
\end{cases}
\end{equation}

Since $c(\psi,0)\neq 0$ and $\sigma(\varphi)^* p = m_0(\varphi, \varphi^{-1}) q U(\varphi)$, it follows from the equations \eqref{pofpsi} and \eqref{qofpsi} that $\sigma(\varphi)^*$ maps $v[\psi,x]$ to $m_0(\varphi, \varphi^{-1}\psi^{-1})$ $v[\psi\varphi,\pi(\varphi)^*x]$ and $v_*[\psi, x_*]$ to $m_0(\varphi, \varphi^{-1} \psi^{-1})$ $v_*[\psi \varphi, \pi_*(\varphi)^*x_*]$. 
Doing the substitutions $\psi \mapsto \psi \varphi^{-1}$, $x\mapsto \pi(\varphi) x$, $x_* \mapsto \pi_*(\varphi) x_*$, we conclude that $\sigma(\varphi)$ is determined by its values on the total set $Z$ of Equation \eqref{totalZ} by the formula 
\begin{equation} \begin{cases} \label{sigma}
\sigma(\varphi) (v[\psi,  x]) &\!\!\!\!\!= m_0(\varphi, \psi^{-1}) v[\psi \varphi^{-1}, \pi(\varphi) x], \\
\sigma(\varphi) (v_*[\psi, x_*]) &\!\!\!\!\!= m_0(\varphi, {\psi}^{-1}) v_*[\psi \varphi^{-1}, \pi_*(\varphi) x_*],
\end{cases}
\end{equation}
for $x\in \mathcal D$, $x_* \in \mathcal D_*$, $\psi \in \mbox{\rm M\"{o}b}$.

Let $M$ be the linear span of $Z$. Define $\sigma_0:\mbox{\rm M\"{o}b} \to \mathcal B(M)$ by $\varphi \mapsto \sigma(\varphi) |_M$. It is immediate from the formula \eqref{sigma} that the graph of $\sigma_0$ is a Borel subset of $\mbox{\rm M\"{o}b}\times \mathcal B(M)$. But, since $M$ is dense in $\mathcal H$,  $\mathcal B(M)$ may be identified with $\mathcal B(\mathcal H)$ by the restriction map $\mathcal B(\mathcal H) \to \mathcal  B(M)$, which is a homeomorphism. The graph of $\sigma_0$ may be identified with the graph of $\sigma$ via this homeomorphism. Thus the graph of $\sigma$ is Borel.  Therefore, as in the proof of Theorem \ref{unirepdil} one concludes that $\sigma$ is a Borel map. Clearly, \eqref{sigma} shows that $\sigma(\rm{id})$ fixes each vector in $Z$. Therefore $\sigma({\rm id}) = I$. 

Let $m$ be the common multiplier of $\pi$ and $\pi_*$. Applying Equation \eqref{reprep} to $\pi$, $\pi_*$ and applying Equation \eqref{mult} to $m_0$, it is trivial to conclude from Equation \eqref{sigma} that the two unitaries $\sigma(\varphi_1 \varphi_2)$ and $(m \cdot m_0)(\varphi_1 , \varphi_2) \sigma(\varphi_1) \sigma(\varphi_2)$ agree on the total set $Z$.  Therefore these two unitaries are equal.  Thus, $\sigma$ is a projective representation with multiplier $m \cdot m_0$ (pointwise product). Since $\sigma$ is associated with $T$, $T$ is an associator. 
\end{proof}

\section{Contractive associators in the Cowen-Douglas classes: the generic case} 

By Theorem \ref{proform}, the characteristic function of any cnu contractive associator is given as a product involving two ``companion representations" $\pi,\pi_*$ of M\"{o}b, and a ``middle operator" $C$.  The object of this section and the next is to demonstrate that, in any concrete case, the explicit determination of this product formula is a highly non-trivial and challenging problem.

Note that, for any operator $T$, $T$ is a cnu contraction if and only if $T^*$ is. Further, $T$ is an associator if and only if $T^*$ is. Indeed, if the representation $\sigma$ is associated with $T$, then $\sigma^\#$ is an associated representation of $T^*$. If $\theta$ is the characteristic function of $T$, then the characteristic function $\theta^*$ of $T^*$ is given by the formula $\theta^*(z) = \theta(\bar{z})^*$, $z\in \mathbb D$. In consequence, if $\pi, \pi_*, C$ are the companions and the middle operator for $T$, then those of $T^*$ are $\pi_*^\#, \pi^\#, C^*$ (respectively).  Thus, the explicit determination of the product formula for $T$ and $T^*$ are equivalent problems. 

We shall say that an associator $T$ is \emph{multiplicity free} if $T$ has an associated representation $\sigma$ which is multiplicity free, i.e., $\sigma$ is a direct sum of mutually inequivalent irreducible projective unitary representations
of M\"{o}b. 

Recall that for positive integer $n$, the Cowen-Douglas class of rank $n$, denoted $B_n(\mathbb D)$, is the class of all bounded linear operators $T$ such that for all $w\in \mathbb D$, $T-wI$ has dense range and a kernel of dimension $n$.  We shall denote by $B_n^*(\mathbb D)$ the class of all operators whose adjoint is in $B_n(\mathbb D)$.  


%

In Theorem \ref{explicit characteristic function} of this section, we determine the product formula for generic contractive associators
in $B_n^*(\mathbb D)$ such that the associated representation is multiplicity free. All multiplicity free irreducible associators in $B_n^*(\mathbb D)$ were described in \cite{KM}. In \cite{ClassCD}, it was shown that upto unitary equivalence these are all. We now describe them explicitly. However, our parametrization is slightly different. The parameters $\mu_i$ here are $\mu_i^2$ in the notation of \cite{KM,ClassCD}.  We have been forced into this re-parametrization by the contingencies of the proofs.

For positive real numbers $\lambda$, let $\mathcal H^{(\lambda)}$ denote the Hilbert space of holomorphic functions on $\mathbb D$ with reproducing kernel $(z,w) \mapsto (1-z\bar{w})^{-\lambda}$.  Let $D_\lambda^+$ denote the holomorphic Discrete series representation of \Mob living on $\mathcal H^{(\lambda)}$  (see \cite{survey}).  It is given by the formula 
$$(D_\lambda^+(\varphi^{-1}) f )(z) = (\varphi^\prime (z) )^{\tfrac{\lambda}{2}} f(\varphi(z)),\, z\in \mathbb D,\,\varphi\in \mbox{\rm M\"{o}b}.$$
Here the branch of the function $(\varphi, z) \mapsto (\varphi^\prime (z) )^{\tfrac{\lambda}{2}}$ is chosen to be Borel in the first argument and analytic in the second. Let $\mathcal{H}_{n}^{(\lambda)}$ denote the Hilbert space  $\displaystyle \oplus_{i = 0}^{n-1} \mathcal{H}^{(\lambda_{i})},$ where $\lambda_i = \lambda + 2i$ and $n\in \mathbb N$. Given an $n$ - tuple of strictly positive real numbers $\underline{\mu}:= (\mu_0, \mu_1, \ldots , \mu_{n-1})$, let 
 $\Gamma^{(\lambda, \underline{\mu})}$ be the map: $\mathcal{H}_{n}^{(\lambda)} \to \mbox{Hol}(\D, \C^{n})$ defined by:
$$ 
\big (\Gamma^{(\lambda,\underline{\mu})}(\underline{f})\big )_\ell = \sum_{0\leq j \leq \ell}\, \frac{\sqrt{\mu_j}\, \binom{\ell}{j}}{(\lambda + 2j)_{\ell-j}} f_j^{(\ell-j)}
$$
for $0 \leq \ell <n$, $f=\oplus_{0\leq j <n} f_j \in \mathcal H_n^{(\lambda)}$.

Here, and in what follows, for a real number $x$ and a non-negative integer $p$, the Pochammer symbol $(x)_p$ is defined by the formula: 
$$(x)_p := \prod_{0\leq i < p} (x + i).$$ 

Finally, let $\mathcal{H}^{(\lambda, \underline{\mu})}$ be the image of $\Gamma^{(\lambda, \underline{\mu})}$.  The map $\Gamma^{(\lambda, \underline{\mu})}$ is injective, see \cite{KM}. Therefore, transplanting via $\Gamma^{(\lambda, \underline{\mu})}$ the inner product from $\mathcal H^{(\lambda)}_n$ makes $\mathcal{H}^{(\lambda, \underline{\mu})}$ 
a reproducing kernel Hilbert space. 
The representation 
$$D^{(\lambda, \underline{\mu})}:=\Gamma^{(\lambda, \underline{\mu})} \Big (\bigoplus_{0\leq i < n} D_{\lambda_i}^+\Big ) \left(\Gamma^{(\lambda, \underline{\mu})}\right)^*$$ 
is a projective unitary representation of \Mob living on $\mathcal{H}^{(\lambda, \underline{\mu})}$. 
The multiplication operators $M^{(\lambda, \underline{\mu})}$ on the Hilbert space $\mathcal{H}^{(\lambda, \mu)}$ are  irreducible associators in $B_{n}^*(\D)$ with associated representation $D^{(\lambda, \underline{\mu})}$.  These are the only irreducible associators in $B_{n}^*(\D)$ whose associated representation is multiplicity free (see \cite[Corollary 4.1]{ClassCD}). Further, $M^{(\lambda, \underline{\mu})}$ and $M^{(\lambda^\prime, \underline{\mu}^\prime)}$ are unitarily equivalent if and only if $\lambda = \lambda^\prime$ and $\underline{\mu}$ and $\underline{\mu}^\prime$ differ by a scalar multiple. 

The first lemma of this section is essentially Lemma 5.1 in \cite{FK}.

\begin{lemma}[Faraut-Koranyi]\label{4.0}
Let $\mathcal H$ be a functional Hilbert space consisting of $\mathbb C^n$ valued functions on a set $\Omega$. Suppose $\mathcal H$ is the orthogonal direct sum of $m$ non-trivial subspaces $\mathcal H_j$ with corresponding reproducing kernels $K_j$ ($0\leq j < m$). Let $a_j$, $0\leq j < m$, be real numbers. Then the kernel $\sum_{0\leq j < m} a_j K_j$ is non-negative definite if and only if $a_j \geq 0$ for all $j$.  
\end{lemma}
\begin{proof}
The ``if'' part is trivial. To prove the ``only if'' part, let $P_j:\mathcal H \to \mathcal H_j$ be the orthogonal projections. Note that the reproducing kernel of $\mathcal H$ is $K:= \sum_{0\leq j < m} K_j $ and we have $P_j K(\cdot,w)\xi = K_j(\cdot,w) \xi$ for $w\in \Omega$, $\xi \in \mathbb C^n$. 

Take $f = \sum_{p=1}^\ell b_p K(\cdot,w_p) \xi_p$, where $b_p \in \mathbb C$, $w_p \in \Omega$, $\xi_p \in \mathbb C^n$. We observe:
$$\sum_{0\leq j < m} a_j \|P_j f\|^2 = \sum_{p,q=1}^\ell b_p \bar{b}_q \big \langle \sum_{0\leq j < m} a_j K_j (w_q, w_p)\xi_p, \xi_q\big \rangle_{\mathbb C^n}\geq 0.$$

Here, the inequality holds because of our assumption that the kernel $\sum_{0\leq j <m} a_j K_j$ is non-negative definite. Since the functions $f$ as above form a dense set in $\mathcal H$, it follows that $\sum_{0\leq j < m} a_j \|P_j f\|^2 \geq 0$ for all $f$ in $\mathcal H$. 
In particular, fixing an index $j$ and taking $f \in \mathcal H_j$, $f\neq 0$, we get $a_j \|f\|^2 \geq 0$. Hence $a_j \geq 0$.  
\end{proof}

\begin{lemma}\label{basic}
Let $\lambda\in \mathbb R$, $\underline{\mu} = 
(\mu_0, \ldots, \mu_{n-1}) \in \mathbb R^n$, where $\mu_0 > 0$.  Then the kernel $B^{(\lambda, \underline{\mu})}:\mathbb D \times \mathbb D \to \mathbb C^{n\times n}$  defined by 
$$
B^{(\lambda, \underline{\mu})}(z,w) = \Big (\!\!\Big ( \sum_{j=0}^{\ell\wedge p} \frac{\binom{\ell}{j}\,\binom{p}{j} \,\mu_j}{(\lambda+2j)_{\ell -j} (\lambda+2j)_{p-j}} \partial^{\ell -j} \bar{\partial}^{p-j} 
(1-z \bar{w})^{-(\lambda+2j)}\Big )\!\!\Big)_{0\leq \ell,p < n}
$$
is non-negative definite if and only if $\lambda \geq 0$ and $\underline{\mu} \geq  \underline{0}$ (entrywise). If $\lambda \geq 0$, $\underline{\mu}\geq \underline{0}$, then $B^{(\lambda, \underline{\mu})}$ is the reproducing kernel of $\mathcal H^{(\lambda, \underline{\mu})}$. 
\end{lemma}
(Here $\partial$ and $\bar{\partial}$ denote partial differentiation with respect to $z$ and $\bar{w}$, respectively.) 
\begin{proof}
Suppose $B^{(\lambda, \underline{\mu})}$ is a non-negative definite kernel. Then each diagonal entry in the matrix defining $B^{(\lambda, \underline{\mu})}$ is a scalar valued non-negative definite kernel on $\mathbb D$. In particular, the kernel $(z,w) \mapsto \mu_0 (1-z\bar{w})^{-\lambda}$, being the top left corner entry of $B^{(\lambda, \underline{\mu})}$, is non-negative definite. Since $\mu_0 > 0$, this forces $\lambda \geq 0$.

Let $\underline{e} = \sum_{0 \leq j < n} \underline{e}_j$, where $\{\underline{e}_j: 0 \leq j < n\}$ is the standard basis of $\mathbb R^n$. Then, by the construction of the functional Hilbert space $\mathcal H^{(\lambda, \underline{e})}$, this space is the orthogonal direct sum of $n$ subspaces with reproducing kernels $B^{(\lambda, \underline{e}_j)}$, $0\leq j < n$. Since $B^{(\lambda, \underline{\mu})}= \sum_{0\leq j < n} \mu_j B^{(\lambda, \underline{e}_j)}$ is assumed to be non-negative definite, it therefore follows from Lemma \ref{4.0}  that $\underline{\mu} \geq \underline{0}$. 

For the converse, note that when $\lambda \geq 0$, $\underline{\mu} \geq \underline{0}$, \cite{KM} shows that $B^{(\lambda, \underline{\mu})}$ is the reproducing kernel of $\mathcal H^{(\lambda, \underline{\mu})}$, and hence it is non-negative definite.  \end{proof}

For integers $k \geq 0$ and real $\lambda^\prime \geq \lambda + 2k$, $\lambda \geq 0$, let $\partial^k: \mathcal H^{(\lambda)} \to \mathcal H^{(\lambda^\prime)}$ denote the bounded operator of $k$ times differentiation. The exact domain and co-domain of any occurrence of this operator should be clear from the context.

\begin{lemma} \label{matrix representation of multiplication operator}
For $0\leq j \leq i,$ define the operator $a_{ij}: \mathcal H^{(\lambda+2j)} \to \mathcal H^{(\lambda+2i)}$ by 
$$a_{ij} = \begin{cases}
M^{(\lambda +2i)}& \mbox{\rm if~} i=j,\\
-\frac{(j+1)_{i-j}}{(\lambda+2j)_{2i-2j-1}} \partial^{i-j-1} & \mbox{\rm if~} i > j.
\end{cases}$$

Let $A : \mathcal{H}_{n}^{(\lambda)} \to \mathcal{H}_{n}^{(\lambda)}$ be the operator $ \left(\Gamma^{(\lambda ,\underline{\mu})}\right)^* M^{(\lambda, \underline{\mu})} \Gamma^{(\lambda ,\underline{\mu})}.$ Then  $A$ admits a block decomposition of the form $A=\left( \!\left(A_{ij}\right) \! \right)_{i,j=0}^{n-1}$, where  $A_{ij}=0$ for $i < j$ 
and $A_{ij} = \sqrt{\frac{\mu_j}{\mu_i}} a_{ij}$ for $i \geq j$.
 \end{lemma}
\begin{proof}
We verify the equality $\Gamma^{(\lambda, \underline{\mu})} A = M^{(\lambda, \underline{\mu})} \Gamma^{(\lambda ,\underline{\mu})}$, where $A$ is the block operator given in this lemma. 
Note that, for $0\leq \ell < n$ and $f=\oplus_{i=o}^{n-1} f_i \in \mathcal H_n^{(\lambda)}$, the $\ell$th component of $\big ( \Gamma^{(\lambda,\underline{\mu})}A - M^{(\lambda, \underline{\mu})} \Gamma^{(\lambda,\underline{\mu})}\big ) f$ is $\sum_{j=0}^{\ell-1}\alpha_j \sqrt{\mu_j} f_j^{(\ell - j -1)}$, where $\alpha_j$ is the difference between the two sides of \eqref{Identity 1} below.  Therefore to complete the proof it suffices to show that: 
\begin{equation}\label{Identity 1}
\displaystyle \sum_{i = j+1}^l \frac{(j+1)_{i-j} \binom{l}{i}}{(\lambda + 2j)_{2i-2j-1} (\lambda + 2i)_{l-i}} = \frac{(l-j)\binom{l}{j}}{(\lambda + 2j)_{l-j}},\,\,0 \leq j < l.
\end{equation}
Note that it is enough to prove the identity \eqref{Identity 1} for $j = 0.$ The general identity then follows form this special case after the substitutions $\lambda \mapsto \lambda+2j$, $i \mapsto i-j$ and $\ell \mapsto \ell-j$.

Now using  the trivial identity 
$$\frac{1}{(\lambda)_{\ell+k}} - \frac{1}{(
\lambda)_{2k+1} (\lambda+2k+2)_{\ell -k -1}} = \frac{\ell-k-1}{(\lambda)_{\ell+k+1}},\,\, 1 \leq k < \ell,$$
it is easy to prove by finite induction on $k$ that 
$$
\sum_{i=1}^k \frac{i!\binom{\ell}{i}}{(\lambda)_{2i-1} (\lambda+2i)_{\ell - i}} = \frac{\ell}{(\lambda)_\ell} - \frac{(\ell-k)_{k+1}}{(\lambda)_{\ell+k}}
$$
for $1\leq k \leq \ell$.  The $j=0$ case of the  identity \eqref{Identity 1} is just the $k=\ell$ case of the last identity. 
%
\end{proof}

\begin{lemma} \label{matrix representation of inclusion map}
For $0\leq j \leq i$, define the operator $b_{ij}: \mathcal H^{(\lambda+2j)} \to  \mathcal H^{(\lambda+2i +1)}$ by $b_{ij} = \frac{(j+1)_{i-j}}{(\lambda+2j)_{2i-2j}} \partial^{i-j}.$
Let $B : \mathcal{H}_{n}^{(\lambda)} \to \mathcal{H}_{n}^{(\lambda+1)}$ be the operator $\left(\Gamma^{(\lambda + 1,\underline{\mu}^\prime)}\right)^*\boldsymbol{i} \Gamma^{(\lambda ,\underline{\mu})},$ where $\boldsymbol i : \mathcal{H}^{(\lambda, \underline{\mu})} \to \mathcal{H}^{(\lambda+1, \underline{\mu}^\prime)}$ is the inclusion map. Then $B$ admits a block decomposition of the form $B = \left(\! \left(B_{ij}\right) \!\right)_{i,j=0}^{n-1}$ where $B_{ij} = 0$ for $ i < j$, and $B_{ij} = \sqrt{\frac{\mu_j}{\mu^\prime_i}} b_{ij}$ for $i \geq j$.  
\end{lemma}

\begin{proof}
We verify the equality $\Gamma^{(\lambda +1 ,\underline{\mu}^\prime)} B =\boldsymbol i \Gamma^{(\lambda ,\underline{\mu})}$, where $B$ is the block operator given above. 
Note that, for $0\leq \ell < n$ and $f=\oplus_{j=0}^{n-1} f_j \in \mathcal H_n^{(\lambda)}$, the $\ell$th component of $\big ( \Gamma^{(\lambda+1,\underline{\mu}^\prime)} B - \boldsymbol i \Gamma^{(\lambda,\underline{\mu})}\big ) f$ is $\sum_{j=0}^\ell \beta_j \sqrt{\mu_j} f_j^{(\ell -j)}$, where $\beta_j$ is the difference between the two sides in \eqref{Identity 2} below. Therefore, to complete the proof, it suffices to show that: 
\begin{equation}\label{Identity 2}
\displaystyle \sum_{i = j}^l \frac{(j+1)_{i-j} \binom{l}{i}}{(\lambda + 2j)_{2i-2j} (\lambda + 2i + 1)_{l-i}} = \frac{\binom{l}{j}}{(\lambda + 2j)_{l-j}},\,\,0 \leq j \leq l.
\end{equation}
Note that to prove the identity \eqref{Identity 2}, it is enough to verify it for the case $j = 0$.  The general case follows from its special case $j=0$ on substituting $i\mapsto i-j$, $\ell \mapsto \ell -j$ and $\lambda \mapsto \lambda +2j$. 

Using  the trivial identity 
$$
\frac{1}{(\lambda)_{\ell+k+1}} - \frac{1}{(\lambda)_{2k+2}((\lambda + 2k +3)_{\ell - k -1}} = \frac{\ell-k-1}{(\lambda)_{\ell+k+2}},\,\, 0 \leq k < \ell, 
$$
it is easy to prove by finite induction on $k$ that , for $0\leq k \leq \ell$, 
$$
\sum_{i=0}^k \frac{i! \binom{\ell}{i}}{(\lambda)_{2i} (\lambda+2i+1)_{\ell -i}} = \frac{1}{(\lambda)_\ell} - \frac{(\ell -k)_{k+1}}{(\lambda)_{\ell+k+1}}.$$
The $j=0$ case of \eqref{Identity 2} is just the case $k=\ell$ of this last identity. 
\end{proof}


\begin{remark}\label{boundedness of inclusion}
Note that Lemma \ref{matrix representation of inclusion map} shows that for all $\lambda \geq 0$ and all $\underline{\mu}, \underline{\mu}^\prime \in \mathbb R_+^n$, the  Hilbert space $\mathcal{H}^{(\lambda, \underline{\mu})}$ is contained in 
$\mathcal{H}^{(\lambda+1, \underline{\mu}^\prime)}$, and the corresponding inclusion map is bounded.  Since the polynomials are dense in all these spaces, it follows that $\mathcal{H}^{(\lambda, \underline{\mu})}$ is densely contained in 
$\mathcal{H}^{(\lambda+1, \underline{\mu}^\prime)}$.
\end{remark}

\begin{lemma}\label{contractivity}
The operator $M^{(\lambda, \underline{\mu})}$ is a contraction if and only if $\lambda \geq 1$ and $\frac{\mu_{k+1}}{\mu_k} \geq \frac{(k+1)^2}{(\lambda + 2k - 1)(\lambda + 2k)}$ for $0 \leq k \leq n-2$.
\end{lemma}

\begin{proof}
Put $\lambda^{\prime\!\prime} = \lambda -1$ and define $\underline{\mu}^{\prime\!\prime} = (\mu_0^{\prime\!\prime}, \ldots \mu_{n-1}^{\prime\!\prime})$ by 
\begin{equation}\label{mudoubleprime} 
\mu_0^{\prime\!\prime} = \mu_0,\, \mu_{k+1}^{\prime\!\prime} = \mu_{k+1} - \frac{(k+1)^2 \mu_{k}}{(\lambda + 2k - 1)(\lambda + 2k)}, \, 0\leq k < n-1.
\end{equation}
Then a computation shows that 
\begin{equation} \label{God}
(1- z\bar{w}) B^{(\lambda, \underline{\mu})}(z,w) = B^{(\lambda^{\prime\!\prime}, \underline{\mu}^{\prime\!\prime})}(z,w). 
\end{equation} 
It is well known that if $\mathcal H$ is a Hilbert space of holomorphic functions on $\mathbb D$ with reproducing kernel $K$, then the multiplication operator $M$ on $\mathcal H$ is a contraction if and only if the kernel $(z,w)\mapsto (1-z\bar{w})K(z,w)$ is non-negative definite.  Therefore, Lemma \ref{basic} implies that $M^{(\lambda,\underline{\mu})}$ is a contraction if and only if $\lambda^{\prime\!\prime}\geq 0, \underline{\mu}^{\prime\!\prime} \geq \underline{0}$.
\end{proof}
Lemma \ref{contractivity} prompts the following definition.
\begin{definition}
The operator $M^{(\lambda, \underline{\mu})}$  is said to be a generic contraction  if $\lambda > 1$ and $\frac{\mu_{k+1}}{\mu_k} > \frac{(k+1)^2}{(\lambda + 2k - 1)(\lambda + 2k)}$ for $0 \leq k \leq n-2.$
\end{definition}

\begin{lemma}\label{existence of unitary U}
Let $M^{(\lambda, \underline{\mu})}$ be a generic contraction.  Let $D$ and $\mathcal D$ be the first defect operator and the first defect space (respectively) of $M^{(\lambda, \underline{\mu})}$. Then there exists  $\underline{\mu}^\prime\in \mathbb R_+^n$  and a unitary operator $U:\mathcal{D} \to \mathcal H^{(\lambda+1, \underline{\mu}^\prime)}$  such that $UD$ is the inclusion map from $\mathcal H^{(\lambda,\underline{\mu})}$ to  $ \mathcal H^{(\lambda+1, \underline{\mu}^\prime)}$. 
\end{lemma}

\begin{proof} 
Let us write $M$ for $M^{(\lambda, \underline{\mu})}$ and $\boldsymbol i:\mathcal H^{(\lambda, \underline{\mu})} \to \mathcal H^{(\lambda+1, \underline{\mu}^\prime)}$ be the inclusion map. Since 
$\boldsymbol i$ has a dense range, it suffices to show that the map $U:Dh \mapsto \boldsymbol i h$ ($h\in \mathcal H^{(\lambda,\underline{\mu})}$) preserves inner product (and hence is well defined) for suitable choice of $\underline{\mu}^\prime$. That is, we must show $\boldsymbol{i}^*\boldsymbol{i}=D^*D\, (= I - M^*M)$.
In view of Lemma \ref{matrix representation of multiplication operator} and Lemma \ref{matrix representation of inclusion map}, it suffices to show that 
$B^*B = I - A^*A$.  Fix indices $0\leq j \leq k < n$. Equating the $(j,k)$th blocks of the two sides, we see that we must prove: 
$$
\sum_{i} B_{ij}^* B_{ik} = \delta_{jk} I - \sum_i A_{ij}^* A_{ik}.
$$

Substituting the formulae for these blocks from Lemma \ref{matrix representation of multiplication operator} and \ref{matrix representation of inclusion map}, we are reduced to proving
$$
\sum_{k \leq i < n} \tfrac{1}{\mu^\prime_i} b_{ij}^* b_{ik} = \tfrac{\delta_{jk}}{\mu_k} I - \sum_{k\leq i < n} \tfrac{1}{\mu_i} a_{ij}^* a_{ik},\,\, 0 \leq j \leq k < n.
$$
Because of the genericity assumption on $M$, we may choose $\underline{\mu}^\prime\in \mathbb R^n_+$ given by 
\begin{equation}\label{***}
\frac{1}{\mu_k^\prime} = \begin{cases} \tfrac{\lambda+2k -1}{\lambda+2k} \tfrac{1}{\mu_k} - \Big (\tfrac{k+1}{\lambda+2k} \Big)^2 \tfrac{1}{\mu_{k+1}}, & 0\leq k \leq n- 2, \\ \tfrac{\lambda + 2k - 1}{\lambda+2k} \tfrac{1}{\mu_k}, & k=n-1.  \end{cases} 
\end{equation}
Substituting these values of $\underline{\mu}^\prime$ in the last equation, we see that, in order to show that this choice of $\underline{\mu}^\prime$ works, we need to prove:
\begin{multline*}
\tfrac{1}{\mu_k} \frac{\lambda+2k -1}{\lambda + 2k} b_{kj}^* b_{kk} + \sum_{k < i < n} \tfrac{1}{\mu_i} \Big ( \frac{\lambda+2i -1}{\lambda+2i} b_{ij}^* b_{ik} - \big ( \frac{i}{\lambda + 2i -2}\big )^2b^*_{i-1,j} b_{i-1,k} \Big ) \\  = \tfrac{1}{\mu_k} \big (\delta_{jk} I - a_{kj}^* a_{kk} \big ) - \sum_{k < i < n} \tfrac{1}{\mu_i} a_{i j}^* a_{i k}. \hskip 7em
\end{multline*}
Note that both sides of this equation are linear combinations of $\tfrac{1}{\mu_i}$, $k \leq i < n$, with operator valued coefficients. Therefore, equating corresponding coefficients, we find that in order to complete the proof we must show that the operators $a_{ij}$, $b_{ij}$ defined in Lemma \ref{matrix representation of multiplication operator} and \ref{matrix representation of inclusion map} satisfy: 
\begin{align*}
a_{kk}^* a_{kk} &= I - \frac{\lambda+2k -1}{\lambda+2k} b_{kk}^*b_{kk},\,\, k \geq 0\\
a_{kj}^* a_{kk} &= - \frac{\lambda+2k -1}{\lambda+2k} b_{kj}^* b_{kk},\,\, 0 \leq j < k,\\
a_{ij}^*a_{ik} &= \big (\tfrac{i}{\lambda+2i-2}\big )^2 b_{i-1, j}^* b_{i-1,k} - \frac{\lambda+2i -1}{\lambda+2i} b_{ij}^* b_{ik},\,\, 0 \leq j \leq k < i.
\end{align*}
For integers $p \geq 0$, let $h_p\in \mathcal H^{(\lambda +2k)}$ be the function defined by $h_p: z \mapsto z^p.$ Since, up to suitable scalar factors, these vectors form an orthonormal basis of $\mathcal H^{(\lambda + 2k)}$, to verify the operator identities given above it suffices to note that both sides map each fixed $h_p$ into the same vector. We omit the elementary details of this verification. 
\end{proof}

\begin{lemma}\label{existence of unitary V}
Let $M^{(\lambda, \underline{\mu})}$ be a generic  contraction. Let $D_*$ and $\mathcal D_*$ denote the second defect operator and the second defect space (respectively) of $M^{(\lambda, \underline{\mu})}$. 
Then there is a $\underline{\mu}^{\prime\!\prime} \in \mathbb R^n_+$ and a unitary operator $V : \mathcal D_* \to \mathcal{H}^{(\lambda-1, \underline{\mu}^{\prime\!\prime})}$  such that $V D_*$ is the adjoint of the   inclusion map from $\mathcal{H}^{(\lambda-1, \underline{\mu}^{\prime\!\prime)}}$ to $\mathcal{H}^{(\lambda, \underline{\mu})}.$ 
\end{lemma}

\begin{proof}
In view of the genericity assumption, $\lambda-1 > 0$ and there exists ${\mu}_i^{\prime\!\prime} > 0$, $0 \leq i \leq n-1$, given by \eqref{mudoubleprime}.
We claim that this choice of $\underline{\mu}^{\prime\!\prime}$ works. 

The sets  $X:= \big \{ B^{(\lambda, \underline{\mu})}(\cdot,w)\zeta:w\in \mathbb D, \zeta\in \mathbb C^n\big \}$ and $Y:= \big \{ B^{(\lambda-1, \underline{\mu}^{\prime\!\prime})}(\cdot,w)\zeta:w\in \mathbb D, \zeta\in \mathbb C^n\big \}$ are total in $\mathcal H^{(\lambda,\underline{\mu})}$ and $\mathcal{H}^{(\lambda-1, \underline{\mu}^{\prime\!\prime})}$ respectively. 
Since by definition, $D_*$ has dense range, it follows that $D_*(X)$ is total in $\mathcal D_*$. The equation  \eqref{God} implies that the map from $D_*(X)$ onto $Y$, given by 
$$
D_*B^{(\lambda, \underline{\mu})}(\cdot,w)\zeta \mapsto B^{(\lambda-1, \underline{\mu}^{\prime\!\prime})}(\cdot,w)\zeta,
$$
preserves inner product. Therefore it extends to a unitary $V$ from $\mathcal D_*$ onto  $\mathcal{H}^{(\lambda-1, \underline{\mu}^{\prime\!\prime})}$. We have 
\begin{align*}
(V D_*)^* B^{(\lambda-1, \underline{\mu}^{\prime\!\prime})}  (\cdot,w)\zeta & =
D_*^*V^*(B^{(\lambda-1, \underline{\mu}^{\prime\!\prime})} (\cdot,w)\zeta)\\
&= D_*^*D_*(B^{(\lambda, \underline{\mu})} (\cdot,w)\zeta)\\
&=B^{(\lambda-1, \underline{\mu}^{\prime\!\prime})} (\cdot,w)\zeta.
\end{align*}
Here the last equality follows again by Equation \eqref{God}. Thus, $(VD_*)^*$ agrees on the total subset $Y$ of $\mathcal{H}^{(\lambda-1, \underline{\mu}^{\prime\!\prime})}$ with the inclusion map from $\mathcal{H}^{(\lambda-1, \underline{\mu}^{\prime\!\prime})}$ to $\mathcal H^{(\lambda, \underline{\mu})}$. Therefore, $V D_*$ is the adjoint of this inclusion.  
\end{proof}

\begin{remark}\label{quasi invertibility}
Lemma \ref{existence of unitary U}  implies that the defect operator $D$ has trivial kernel. In other words, the generic contractions $M^{(\lambda, \underline{\mu})}$ 
are pure contractions.  In consequence, they are cnu contractions. Therefore, the theory developed in the previous sections applies to them.  Since $\mathcal H^{(\lambda - 1, \underline{\mu}^{\prime\!\prime})}$ is densely contained in $\mathcal H^{(\lambda, \underline{\mu})}$, Lemma \ref{existence of unitary V} implies that $D_*$ also has trivial kernel. Therefore, the adjoint of the generic contraction $M^{(\lambda, \underline{\mu})}$ is also a pure contraction. In consequence, for a generic contraction $M=M^{(\lambda, \underline{\mu})}$, the operators $(I-M^*M)^{1/2}$ and $(I-MM^*)^{1/2}$ have trivial kernels. Since these operators are self adjoint, it follows that they have dense range. Hence, the defect spaces of this generic contraction are $\mathcal D = \mathcal H^{(\lambda, \underline{\mu})}=\mathcal D_*$.
\end{remark}

\begin{lemma} \label{realization left and right representations}
Let $\pi, \pi_*$ be the representations of \Mob occurring in the product formula for the characteristic function of the generic contraction $M^{(\lambda, \underline{\mu})}$.  Also, let $\underline{\mu}^\prime$, $\underline{\mu}^{\prime\!\prime}$, $U, V,$ be as in Lemma \ref{existence of unitary U} and Lemma \ref{existence of unitary V}. Then $U \pi(\varphi) U^* =  D^{(\lambda+1, \underline{\mu}^\prime)}( \varphi )$ and $V \pi_*(\varphi) V^* =  D^{(\lambda-1, \underline{\mu}^{\prime\!\prime)}}(\varphi)$, $\varphi \in \mbox{M\"{o}b}$.
\end{lemma}

\begin{proof}
Let $\boldsymbol i^+: \mathcal H^{(\lambda, \underline{\mu})} \to \mathcal H^{(\lambda+1, \underline{\mu}^\prime)}$ and $\boldsymbol i^-: \mathcal H^{(\lambda-1, \underline{\mu}^{\prime\!\prime)}} \to \mathcal H^{(\lambda, \underline{\mu})}$  be the respective inclusion maps. Thus, $\boldsymbol i^+ = UD, \, \boldsymbol i^- = D_*^* V^*$.  

We recall from \cite{KM} that, for $\lambda > 0, \underline{\mu} \in \mathbb R_+^n$, there is a function $J^{(\lambda)}: \mbox{\Mob} \times \mathbb D \to \mathbb C^{n\times n}$ such that 
$$
\big (D^{(\lambda, \underline{\mu})}(\varphi^{-1}) f\big ) (z) = J^{(\lambda)}(\varphi, z) f(\varphi z) 
$$
for $f\in \mathcal H^{(\lambda, \underline{\mu})}, z\in \mathbb D, \varphi \in \mbox{Mob}$. 
The explicit formula for $J^{(\lambda)}$ (available in \cite{KM}) does not concern us here. It suffices to note that $J^{(\lambda)}$ depends only on $\lambda$ and $n$ (and not on $\underline{\mu}$), and it satisfies 
\begin{align*}
J^{(\lambda+1)} (\varphi,z) &= c(\varphi , z) J^{(\lambda)}(\varphi,z), \,\, \lambda > 0, \\
J^{(\lambda-1)} (\varphi,z) &= c(\varphi , z)^{-1} J^{(\lambda)}(\varphi,z), \,\,\lambda > 1.
\end{align*}
Using the formula for $\pi$ from Theorem \ref{tworeps} (with $\sigma = D^{(\lambda, \underline{\mu})}, T = M^{(\lambda , \underline{\mu})}$), we get, for $f\in \mathcal H^{(\lambda , \underline{\mu})}$, 
\begin{align*}
(U\pi(\varphi)U^*)(\boldsymbol i^+ f) &= U \pi(\varphi) D f \\
&= m_0(\varphi, \varphi^{-1})UD D^{(\lambda, \underline{\mu})}(\varphi) \big  (c(\varphi, M^{(\lambda, \underline{\mu})} )^{-1}f\big )\\
&= m_0(\varphi, \varphi^{-1}) \boldsymbol i^+ D^{(\lambda, \underline{\mu})}(\varphi) \big  (c(\varphi, M^{(\lambda, \underline{\mu})} )^{-1}f\big )\\
&= D^{(\lambda+1, \underline{\mu}^\prime)}(\varphi) (\boldsymbol i^+ f).
\end{align*}
Since $\boldsymbol i^+$ has dense range (see Remark \ref{boundedness of inclusion}), this shows that $U\pi(\varphi)U^*= D^{(\lambda+1, \underline{\mu}^\prime)}( \varphi )$. (Here, for the last equality in the above string, we have used the identity $c(\varphi, \varphi^{-1} z))^{-1} = m_0(\varphi, \varphi^{-1}) c(\varphi^{-1},z)$ from Equation \eqref{pm 1} and the relation between $J^{(\lambda)}$ and $J^{(\lambda+1)}$ noted above.)

Next, the formula for $\pi_*$ from Theorem \ref{tworeps}  (with $\sigma = D^{(\lambda, \underline{\mu})}, T = M^{(\lambda, \underline{\mu})}$) may be manipulated to yield $\pi_*(\varphi)^* D_* = m_0(\varphi, \varphi^{-1}) D_* c(\varphi, {M^{(\lambda, \underline{\mu})}})^* D^{(\lambda, \underline{\mu})}(\varphi)^*$. Hence, taking adjoints, we  get 
$$D_*^* \pi_*(\varphi) = m_0(\varphi, \varphi^{-1}) D^{(\lambda, \underline{\mu})}(\varphi) c(\varphi, M^{(\lambda, \underline{\mu})} ) D_*^*.$$
Hence we have for $g \in \mathcal H^{(\lambda - 1, \underline{\mu}^{\prime\!\prime})},$ $z\in \mathbb D$,
\begin{align*}
(V\pi_*(\varphi)V^*g)(z) &= (\boldsymbol i^- V\pi_*(\varphi) V^* g)(z) \\
&= (D_*^*\pi_*(\varphi)V^*g)(z)\\
&= m_0(\varphi, \varphi^{-1}) \big ( D^{(\lambda, \underline{\mu})}(\varphi)c(\varphi, M^{(\lambda, \underline{\mu})}) D_*^* V^* g\big )(z)\\
&= m_0(\varphi, \varphi^{-1}) \big ( D^{(\lambda, \underline{\mu})}(\varphi)c(\varphi, M^{(\lambda, \underline{\mu})}) \boldsymbol i^- g\big )(z)\\
&=  (D^{(\lambda-1, \underline{\mu}^{\prime\!\prime})}(\varphi) g) (z).\end{align*} 
Thus $V\pi_*(\varphi)V^* =  D^{(\lambda-1, \underline{\mu}^{\prime\!\prime})}(\varphi)$. (Here again, for the last equality in the above string, we have used the identity $c(\varphi, \varphi^{-1} z) = m_0(\varphi, \varphi^{-1}) c(\varphi^{-1}, z)^{-1}$ and the formula for $J^{(\lambda-1)}$ in terms of $J^{(\lambda)}$.)  
\end{proof}

\begin{lemma}\label{computation of C}
Let $M^{(\lambda, \underline{\mu})}$ be a generic contraction. Let $C : \mathcal{H}_n^{(\lambda + 1)} \to \mathcal{H}_n^{(\lambda - 1)}$ be the operator defined by $C = -{\Gamma^{(\lambda - 1, \underline{\mu}^{\prime\!\prime})}}^* V{M^{(\lambda, \underline{\mu})}} U^* \Gamma^{(\lambda + 1, \underline{\mu}^{\prime})}$, where $\underline{\mu}^\prime$, $\underline{\mu}^{\prime\!\prime}$, $U$ and $V$ are as in Lemma \ref{existence of unitary U} and \ref{existence of unitary V}. Then $C=\left(\!\left(x_{jk}(\partial^{k-j+1})^*\right)\!\right)_{0 \leq j, k < n}$ where the real matrix $\left(\!\left(x_{jk}\right)\!\right)_{0 \leq j,k < n}$ is given by   
\[   
x_{jk} = 
     \begin{cases}
       0 & \mbox{\rm if}\,\, j > k + 1,\\
       \vspace*{0.21cm}
       + \frac{ (k + 1)\mu_k}{\sqrt{\mu^\prime_k \mu^{\prime\!\prime}_{k + 1}} (\lambda + 2k - 1)}  & \mbox{\rm if}\,\, j = k + 1,\\
      - \frac{ \sqrt{\mu^{\prime\!\prime}_j} (j + 1)_{k - j}}{\sqrt{\mu_k^\prime} (\lambda + 2j - 1)_{2k-2j +1}} & \mbox{\rm if}\,\, j < k + 1.
     \end{cases}
\]
\end{lemma}

\begin{proof}
Let $i^{+} : \mathcal{H}^{(\lambda, \underline{\mu})} \to \mathcal{H}^{(\lambda+1, \underline{\mu}^\prime)}$, $i^- : \mathcal{H}^{(\lambda-1, \underline{\mu}^{\prime\!\prime})} \to \mathcal{H}^{(\lambda, \underline{\mu})}$ be the inclusion maps. Define $A = {\Gamma^{(\lambda, \underline{\mu})}}^* M^{(\lambda, \underline{\mu})} \Gamma^{(\lambda, \underline{\mu})}$, $B^+ =  {\Gamma^{(\lambda+1, \underline{\mu}^\prime)}}^* i^+ \Gamma^{(\lambda, \underline{\mu})}$ and $B^- =  {\Gamma^{(\lambda, \underline{\mu})}}^* i^- \Gamma^{(\lambda-1, \underline{\mu}^{\prime\!\prime})}$. Note that, by Remark \ref{quasi invertibility}, $\mathcal{H}^{(\lambda, \underline{\mu})}$ is the co-domain of both the defect operators of $M^{(\lambda, \underline{\mu})}$, hence it is the domain of both $U$, $V$. This is why the product defining $C$ makes  sense. Since $U$, $V$, $\Gamma^{(\cdot, \cdot)}$ are unitaries, using the formulae $i^+ = UD$, $i^- = D_*^* V^*$, $M^{(\lambda, \underline{\mu})} D^* = D_*^* M^{(\lambda, \underline{\mu})}$, we routinely derive the equation
\begin{equation}\label{eqn 4.12}
B^- C = -A {B^+}^*.
\end{equation}
Note that the $n \times n$ block decompositions of the operators $A$, $B^\pm$ are given by Lemma \ref{matrix representation of multiplication operator} and Lemma \ref{matrix representation of inclusion map}. Since $i^-$ has trivial kernel, so has $B_-$. Therefore, the operator $C$ is uniquely determined by Equation \eqref{eqn 4.12}. Therefore, to complete the proof, it suffices to verify that the block operator $C$ given in the statement of this lemma is a solution to Equation \eqref{eqn 4.12}.

For $0\leq j \leq i$, let $a_{ij}$ be the operator from Lemma \ref{matrix representation of multiplication operator},  let $b^+_{ij}$ be the operator which we called $b_{ij}$ in Lemma \ref{matrix representation of inclusion map} and let $b^-_{ij}$ be the operator obtained from $b^+_{ij}$ by the substitution $\lambda \mapsto \lambda -1$ in its description. For $0\leq j \leq k+1$, let  $c_{jk}: \mathcal H^{(\lambda+2k +1)} \to \mathcal H^{(\lambda+2j-1)}$ be the operator defined by:
$$
c_{jk}=\begin{cases} 
\frac{k+1}{\lambda+2k-1} I & \mbox{\rm ~if ~} j = k+1,\\
- \frac{(j+1)_{k-j}}{(\lambda+2j-1)_{2k-2j+1}} (\partial^{k-j+1})^* & \mbox {\rm ~ if~} j < k+1.
\end{cases}
$$
Thus the blocks of the $n\times n$ block operators $A$, $B^\pm$, $C$ are given by 
\begin{align*}
A_{ij} &= \begin{cases}
0 & \mbox{\rm ~if~} j > i,\\
\sqrt{\frac{\mu_j}{\mu_i}} a_{ij} &\mbox{\rm ~if~} j \leq i.
\end{cases}\\
B_{kj}^+ & =  \begin{cases}
0 &\mbox{\rm ~if~}  j >  k,\\
\sqrt{\frac{\mu_j}{\mu^\prime_k}} b^+_{kj} & \mbox{\rm ~if~}  j \leq k.
\end{cases}\\
B_{ij}^- &=  \begin{cases}       
0 &\mbox{\rm ~if~} j > i, \\
\sqrt{\frac{\mu_j^{\prime\!\prime}}{\mu_i}}b^-_{ij} &\mbox{\rm ~if~}  j \leq i.
\end{cases}\\
C_{jk} &=  \begin{cases}
0 &\mbox{\rm ~if~}  j > k+1,\\
\frac{\mu_k}{\sqrt{\mu_k^\prime \mu_{k+1}^{\prime\!\prime}}}c_{jk} &\mbox{\rm ~if~}  j = k+1,\\
\sqrt{\frac{\mu_j^{\prime\!\prime}}{\mu_k^\prime}} c_{j k} &\mbox{\rm ~if~} j < k+1.
\end{cases}
\end{align*}
Fix indices $0 \leq i , k < n$. To prove the equation \eqref{eqn 4.12}, it suffices to equate the (i,k)th blocks of its two sides. That is, we must show that $\sum_j B_{ij}^- C_{jk} = - \sum_j A_{ij} {B_{kj}^+}^*$. In view of the preceding formulae, this reduces to: 
$$\mu_k \delta_{i > k} b_{i,k+1}^- c_{k+1, k} + \sum_{0 \leq j \leq i\wedge k} \mu_j^{\prime\!\prime} b_{ij}^-c_{jk} = - \sum_{0 \leq j \leq i \wedge k} \mu_j a_{ij} {b_{kj}^+}^*.$$ 
Here we have used the following variation of  the Kronecker delta, $\delta_{i > k} : = \begin{cases} 1 \mbox{\rm ~if~}i > k, \\ 0 \mbox{\rm ~if~} i \leq k\end{cases}$. 
Substituting the formulae  for $\mu_{\mathbf{\cdot}}^{\prime\!\prime}$ from Equation \eqref{mudoubleprime}, this in turn reduces to (with $\ell:= i \wedge k$) 
$$
\mu_\ell \big ( b_{i\ell}^- c_{\ell k} + \delta_{ i > k} b_{i, k+1}^- 
c_{k+1,k}\big )  + \sum_{0\leq j < \ell} \mu_j \big ( b_{ij}^- c_{jk} - \frac{(j+1)^2}{(\lambda + 2j -1)(\lambda+2j)} b_{i,j+1}^-c_{j+1, k}\big ) = - \sum_{0\leq j \leq \ell} \mu_j a_{ij} {b_{kj}^+}^*.
$$
Note that both sides here are linear combinations of $\mu_j$, $0 \leq j \leq i \wedge k$, with operator coefficients.  Equating corresponding coefficients, we see that, to complete the proof, we need to verify the following operator identities: 
\begin{align*}
a_{ii} {b_{ki}^+}^* &= - b_{ii}^- c_{ik}, \,\, 0\leq i \leq k.\\
a_{ik} {b_{kk}^+}^* &= - b_{ik}^- c_{kk} - b_{i,k+1}^- c_{k+1, k}, \,\, 0 \leq k < i.\\
a_{ij} {b_{kj}^+}^* &= \frac{(j+1)^2}{(\lambda+2j-1)(\lambda+2j)} b_{i,j+1}^- c_{j+1,k} - b_{ij}^- c_{jk},\,\, 0 \leq j < i\wedge k.
\end{align*}
This verification may be done as in the proof of Lemma \ref{existence of unitary U}. We omit the details. 
\end{proof}

For $\lambda > 0$, let $M^{(\lambda)}$ be the operator of multiplication by the coordinate function on the Hilbert space $\mathcal H^{(\lambda)}$ with reproducing  kernel $K^\lambda(z,w) = (1-z\bar{w})^{-\lambda}$ defined on $\mathbb D\times \mathbb D$. The operator  ${M^{(\lambda)}}$ is an associator in the class  $B_1^*(\D)$.  Its associated representation is $D_{\lambda}^+$, the holomorphic Discrete series representation of M\"{o}b on the Hilbert space $\mathcal{H}^{(\lambda)}$. These are all the unitarily 
inequivalent associators in $B_1^*(\mathbb D)$ (cf. \cite{Mi}). 

It was shown in \cite[Theorem 3.1]{scal}  that the characteristic function of the homogeneous contraction $M^{(\lambda)},$ $\lambda > 1,$ coincides with the purely contractive holomorphic function $\theta_{\lambda} : \mathbb{D} \longrightarrow \mathcal{B}(\mathcal H^{(\lambda + 1)}, \mathcal H^{(\lambda - 1)}),$ where 
\begin{equation}\label{prod1}
\theta_{\lambda}(z) = \frac{1}{\sqrt{\lambda ( \lambda - 1)}} D_{\lambda - 1}^+(\varphi_{z})^{*}\partial^{*}D_{\lambda + 1}^+(\varphi_{z}).
\end{equation}

Here $\partial : \mathcal H^{(\lambda-1)} \to \mathcal H^{(\lambda + 1)}$ is the derivation: $\partial f = f'.$ 
\begin{theorem}\label{explicit characteristic function}
Let $M^{(\lambda, \underline{\mu})}$ be a generic contraction. Then the characteristic function of $M^{(\lambda, \underline{\mu})}$ coincides with
the function $\theta^{(\lambda, \underline{\mu})}: \mathbb D \to \mathcal B \big ( 
\oplus_{0\leq k < n} \mathcal H^{(\lambda + 2k +1)},  \oplus_{0\leq j < n} \mathcal H^{(\lambda + 2j - 1)}\big )$ given by the formulae
\begin{align*}
\theta^{(\lambda, \underline{\mu})}(z) &= \Big (\bigoplus_{0\leq j < n} D^+_{\lambda + 2j -1} (\varphi_z)^*\Big ) C \Big ( \bigoplus_{0\leq k < n} D^+_{\lambda +2k + 1}(\varphi_z)\Big )\\
&= \big (\!\!\big ( \theta_{j\,k} (z) \big ) \!\!\big )_{0\leq j,k<n}, \,\, z\in \mathbb D,
\end{align*}
where $C$ is the block operator given by Lemma \ref{computation of C} and 
$$
\theta_{j\,k}(z) = \begin{cases} 0 & \mbox{~\rm if~} j > k+1,\\
y_{j k}  \displaystyle{\prod_{j\leq \ell \leq k}} \theta_{\lambda+2\ell }(z) & \mbox{~\rm if~} j \leq k+1,
\end{cases}
$$ 
with $y_{jk} = x_{jk} \sqrt{(\lambda+2j -1)_{2k-2j+2}},\,\, j \leq k+1$.
\end{theorem}
(Here the constants $x_{jk}$ are as in Lemma \ref{computation of C}, and the factors in the second formula are given by Equation \eqref{prod1}. As usual, the empty product (which occurs when $j=k+1$) denotes the identity.)
\begin{proof}
In view of Remark \ref{quasi invertibility}, the product formula for the characteristic function $\theta$ of $M^{(\lambda, \underline{\mu})}$ takes the form $\theta(z) = - \pi_*(\varphi_z)^* M^{(\lambda, \underline{\mu})} \pi(\varphi_z)$. Define $\theta^{(\lambda, \underline{\mu})}$ by 
$$
\theta^{(\lambda, \underline{\mu})}(z) = {\Gamma^{(\lambda-1, \underline{\mu}^{\prime\!\prime})}}^* V \theta(z) U^*\Gamma^{(\lambda+1, \underline{\mu}^\prime)}, \,z\in \mathbb D. 
$$
Since $\Gamma^{(\cdot, \cdot)}$, $U,V$ are unitaries, it follows that $\theta$ coincides with $\theta^{(\lambda, \underline{\mu})}$. The first formula for $\theta^{(\lambda, \underline{\mu})}$ (in this theorem) is now immediate from Lemmas \ref{realization left and right representations} and \ref{computation of C}. Doing this block multiplication, we obtain $\theta^{(\lambda, \underline{\mu})}(z) = \big (\!\! \big ( \theta_{j\\,k}(z) \big ) \!\!\big )$, where $\theta_{j\,k}(z) = 0$ when $j > k+1$ and, when $j \leq k+1$, 
$$
\theta_{j\,k}(z) = x_{j\,k}D^+_{\lambda+2j -1}(\varphi_z)^* (\partial^{k-j+1})^* D^+_{\lambda+2k+1}(\varphi_z).
$$
But we have, for $j \leq k +1$, 
\begin{align*}
D^+_{\lambda+2j -1}(\varphi_z)^* (\partial^{k-j+1})^* D^+_{\lambda+2k+1}(\varphi_z)&= \prod_{j\leq \ell\leq k}\big ( D^+_{\lambda + 2 \ell -1} (\varphi_z)^* \partial^* D^+_{\lambda+2\ell +1} (\varphi_z)\big )\\
& = \prod_{j\leq \ell \leq k}\big (\sqrt{(\lambda+2\ell)(\lambda+2\ell -1)} \theta_{\lambda+2 \ell}(z) \big )\\
&= \sqrt{(\lambda+2j -1)_{2k-2j+2}} \prod_{j\leq \ell \leq k} \theta_{\lambda+2\ell}(z). 
\end{align*}
This completes the proof.
\end{proof}

\section{Contractive associators in the Cowen-Douglas classes: an extremal case}
In Theorem \ref{explicit characteristic function} of the last section, we obtained the explicit product formula for the generic irreducible multiplicity free contractive associators in $B_n^*(\mathbb D)$.  In this section, we do the same in the extreme opposite case, namely, we look at the irreducible multiplicity free contractive associators in $B_n^*(\mathbb D)$, which are the most non-generic, in a definite sense. Thus we introduce: 
\begin{notation}
For positive integers $n$, and real numbers $\lambda > 1$, let $M_{\lambda,n}$ denote the operator $M^{(\lambda, \underline{\mu})} \in B_n^*(\mathbb D)$, where $\underline{\mu} \in \mathbb R_+^n$ is given by $\mu_k = \frac{k!^2}{(\lambda-1)_{2k}}$, $0\leq k < n$. 
\end{notation} 
In other words, $M_{\lambda, n}$ is the operator $M^{(\lambda,\underline{\mu})}$ with $\lambda > 1,\, \mu_0 =1$ and $\mu_{k+1} = \frac{(k+1)^2\mu_k}{(\lambda+2k-1)(\lambda+2k)}$ for $0\leq k < n-1$. Thus by Lemma \ref{contractivity}, $M_{\lambda,n}$ is a contraction. But $M_{\lambda, n}$ are the only multiplicity free irreducible contractive associators in $B_n^*(\mathbb D)$ which violate all the 
requirements (except the inequality $\lambda > 1$) in the definition of genericity. 
The main result of this section is: 
\begin{thm}\label{theorem 5.2}
For real numbers $\lambda > 1,$ integers $n \geq 1$, $M_{\lambda,n}$ is a cnu contraction, and its characteristic function coincides with the function $\theta_{\lambda, n}: \mathbb D \to \mathcal B \big (\mathcal H^{(\lambda+2n-1)}, \mathcal H^{(\lambda-1)}\big )$ given by 
\begin{align*}
\theta_{\lambda,n}(z) &= \frac{1}{\sqrt{(\lambda-1)_{2n}}} D^+_{\lambda-1}(\varphi_z)^* (\partial^n)^*D^+_{\lambda+2n -1}(\varphi_z)\\
&=\prod_{0\leq k < n} \theta_{\lambda+2k}(z), \,\, z\in \mathbb D.
\end{align*} 
\end{thm}
(Here, again, the factors in the second formula are given by Equation \eqref{prod1}.)  

The rest of this section is devoted to a proof of this theorem. For $\lambda > 0$, we identify the Hilbert space $\mathcal H^{(\lambda)}\otimes H^2$ with a Hilbert space of holomorphic functions on the bidisc $\mathbb D^2$ via the map 
$$ f \otimes g \mapsto ( (z,w) \mapsto g(z)f(w) ), \,\, f \in \mathcal H^{(\lambda)}, g \in H^2.$$ 
For $p \geq 0$, let $Hom(p)$ denote the vector space of all homogeneous polynomials of degree $p$ in two complex variables $z,w$. Note that we have the orthogonal decomposition
$$\mathcal H^{(\lambda)} \otimes H^2 = \bigoplus_{p \geq 0} Hom(p).$$ 
Let $\triangle:=\{(z,z):z\in \mathbb D\}$, the disc diagonally embedded in the bidisc. For $\lambda > 0,$ $k \geq 0$, let $V_{k,\lambda}$ denote the maximal subspace of $\mathcal H^{(\lambda)} \otimes H^2$ which is orthogonal to all $h\in \mathcal H^{(\lambda)}\otimes H^2$ such that $h$ vanishes to order $\geq k$ on $\triangle$. Define $V_{k,\lambda}(p):=Hom(p) \cap V_{k,\lambda}$. Note that we have the filtration 
\begin{equation} \label{filtration}
\{0\} = V_{0,\lambda}(p) \subseteq V_{1,\lambda}(p) \subseteq \cdots \subseteq V_{p+1, \lambda}(p) = Hom(p)
\end{equation}
as well as the orthogonal decomposition 
\begin{equation}\label{grading}
V_{k,\lambda} = \bigoplus_{p\geq 0} V_{k,\lambda}(p).
\end{equation}
For $0\leq j \leq p$, define the polynomial $h^\lambda_{j,p} \in Hom(p)$ by 
$$
h^\lambda_{j,p}(z,w) := \sum_{j \leq i \leq p} \binom{i}{j} \binom{p-i+\lambda-1}{p-i}z^i w^{p-i}.
$$
\begin{lem}\label{lemma 5.3}
For $0\leq k \leq p+1$, the set $\{h^\lambda_{j,p}: 0\leq j < k \}$ is a basis for $V_{k,\lambda}(p)$. 
\end{lem}
\begin{proof}
Consider the $(p+1)\times k$ matrix $D= \big (\!\!\big (\binom{i}{j} \big ) \!\!\big )_{{0\leq i \leq p}\atop{0\leq j < k}}.$  As usual, we view $D$ as a linear operator from $\mathbb C^k$  into $\mathbb C^{p+1}$, with the standard inner products. A straightforward calculation shows that a  polynomial $\sum_{0\leq i \leq p} a_i z^i w^{p-i}  \in Hom(p)$ is orthogonal to $V_{k, \lambda}$ if and only if the vector $\underline{a}$ belongs to the kernel of $D^*$. Therefore, a polynomial   $\sum_{0\leq i \leq p}b_i z^i w^{p-i}  \in Hom(p)$ is in $V_{k, \lambda}$ if and only if it is orthogonal in $\mathcal H^{(\lambda)}\otimes H^2$ to $\sum_{0\leq i \leq p} a_i z^i w^{p-i}$ for all $\underline{a}$ in $\ker D^*$, i.e., if and only if, 
$$\sum_{0\leq i \leq p}\frac{\bar{a}_i b_i}{\binom{p-i+\lambda-1}{p-i}}=0, \,\, \forall \underline{a} \in \ker D^*.$$
That is, $\sum_{0\leq i \leq p} b_iz^i w^{p-i}$ is in $V_{k, \lambda}$ if and only if the vector $\Big (\tfrac{b_i}{\tbinom{p-i+\lambda-1}{p-i}}\Big )_{0\leq i \leq p}$ is orthogonal to $\ker D^*$ in $\mathbb C^{p+1}$, i.e., if and only if this vector  belongs to $Im\, D$. Since the $k$ columns of $D$ form  a basis of $Im\, D$, the result follows. \end{proof}

For $\lambda  > 1$, let $\Theta_\lambda: \mathcal H^{(\lambda+1)} \otimes H^2 \to \mathcal H^{(\lambda -1)} \otimes H^2$ be the characteristic operator corresponding to the characteristic function of $M^{(\lambda)}$. The following formula for the action of $\Theta_\lambda^*$ is from \cite{scal}: 
%
%
\begin{equation}\label{Theta*}
\left(\Theta_{\lambda}^{*} f\right) (z, w) = \frac{1}{\sqrt{\lambda (\lambda - 1)}}\frac{\partial}{\partial w}f(z, w) - \sqrt{\frac{\lambda - 1}{\lambda}} \frac{f(z, w) - f(w, w)}{z - w},\,z, w \in \mathbb{D},
\end{equation}
for $f \in \mathcal H^{(\lambda - 1)}\otimes H^2$. 

\begin{lem}\label{Decomposition of Theta}
For $\lambda > 1,$ $n\geq 1$, the operator $\Theta_{\lambda}^{*}$ maps $V_{n, \lambda-1}$ into $V_{n-1, \lambda+1}$ and $V^\perp_{n,\lambda-1}$ into $V^\perp_{n-1, \lambda+1}$.
\end{lem}
\begin{proof} Using the formula \eqref{Theta*}, it is easy to see that $f\in V^\perp_{n,\lambda-1}$ implies $\Theta^*_\lambda f \in V^\perp_{n-1, \lambda+1}$. Thus 
$\Theta_{\lambda}^{*}$ maps $V^\perp_{n, \lambda-1}$ into $V^\perp_{n-1, \lambda+1}$. In view of the decomposition \eqref{grading}, in order to prove that $\Theta^*_\lambda$ maps $V_{n, \lambda-1}$ into $V_{n-1, \lambda+1}$, it suffices to show that it maps $V_{n, \lambda-1}(p)$ into $V_{n-1, \lambda+1}(p-1)$ for all $p$. Since  $V_{n-1, \lambda+1}(p-1)= Hom(p-1)$ for $p<n$, and since it is clear from equation \eqref{Theta*} that $\Theta_\lambda^*$ maps $Hom(p)$ into $Hom(p-1)$, it is enough to fix $p \geq n$ and show that $\Theta^*_\lambda$ maps $V_{n, \lambda-1}(p)$ into 
$V_{n-1, \lambda+1}(p-1)$ for this $p$. In view of Lemma \ref{lemma 5.3} and the filtration \eqref{filtration}, it suffices to show that 
\begin{equation} \label{5.19}
\Theta_\lambda^*(h_{j,\, p}^{\lambda -1} ) \in V_{j, \lambda+1} (p-1), \,\, 0 \leq j \leq p.
\end{equation} 


Using equation \eqref{Theta*} it is not hard to prove that  $\Theta_\lambda^*(h_{0,\,p}^{\lambda-1}) = 0$ (and hence \eqref{5.19} is true for $j=0$), and 
$$\sqrt{\frac{\lambda}{\lambda-1}}(z-w)\Theta_\lambda^* (h^{\lambda-1}_{j,\, p}) = \binom{\lambda + p - 1}{p-j} z^0 w^p - \sum_{1\leq i \leq p} \binom{i-1}{j-1} \binom{\lambda+p-i-1}{p-i} z^i w^{p-i}, \,\, 1 \leq j \leq p.$$
Using this formula, it is easy to verify that, for $0 \leq j < p$, there is a real number $c$ (depending on $j, p, \lambda$) such that 
$$
\Theta_\lambda^*(h_{j+1, p}^{\lambda-1} ) / \binom{\lambda+p-1}{p-j-1} - \Theta_\lambda^*(h_{j,\,p}^{\lambda -1}) / \binom{\lambda+p-1}{p-j} = c\, h_{j, p-1}^{\lambda+1}. 
$$
(Namely, to verify this, multiply both sides by $z-w$ and use the previous equation.) 
Using Lemma \ref{lemma 5.3} and the filtration \eqref{filtration}, it is now trivial to prove \eqref{5.19} by finite induction on $j$.
\end{proof}
Now, for $\lambda > 1,$ $n\geq 1,$ let $\theta_{\lambda, n}$ be as in Theorem \ref{theorem 5.2}:
$$\theta_{\lambda,n}(z):=\prod_{0\leq k < n} \theta_{\lambda+2k}(z), \,z\in \mathbb D.$$
Since $\theta_{\lambda,n}$ is a finite pointwise product of pure contraction valued holomorphic functions, it follows that $\theta_{\lambda,n}$ is a pure contraction valued holomorphic function on $\mathbb D$.  Therefore, it is a characteristic function. Let $\Theta_{\lambda,n}:\mathcal H^{(\lambda+2n-1)}\otimes H^2 \to \mathcal H^{(\lambda -1)}\otimes H^2$ be the characteristic operator corresponding to $\theta_{\lambda, n}$. 
\begin{lem}\label{lemma 5.5}
For $\lambda > 1,$ $n\geq 1,$ the kernel of $\Theta^*_{\lambda, n}$ is $V_{n,\lambda-1}$.
\end{lem}
\begin{proof}
By \cite[Theorem 3.2]{scal}, $\ker \Theta^*_\lambda = V_{1, \lambda-1}$ . Since $\theta_{\lambda,1}  = \theta_\lambda,$ this proves the result for $n=1$. Now, let $n >1$. Then $\theta_{\lambda, n}= \theta_{\lambda} \theta_{\lambda+2, n-1}$ (pointwise product) and hence $\Theta_{\lambda, n} = \Theta_\lambda \Theta_{\lambda+2, n-1}$.  Therefore, $\ker (\Theta^*_{\lambda,n}) = {\Theta^*_\lambda}^{-1} \big ( \ker (\Theta_{\lambda+2,n-1})\big )$ for $n > 1$. Hence, to complete the proof by induction on $n$, it suffices to show that  ${\Theta^*_\lambda}^{-1}(V_{n-1, \lambda+1}) = V_{n,\lambda-1}$. 

Let $f\in {\Theta^*_\lambda}^{-1} (V_{n-1, \lambda+1})$. That is, $\Theta^*_\lambda f \in V_{n-1, \lambda+1}$. Write $f = g+h$, where $g\in V_{n,\lambda-1},$ $h\in V^\perp_{n,\lambda-1}$. By Lemma \ref{Decomposition of Theta}, $\Theta^*_\lambda h \in V^\perp_{n-1, \lambda+1}$ and $\Theta^* _\lambda g \in V_{n-1, \lambda+1},$ so that $\Theta^*_\lambda h = \Theta^*_\lambda f - \Theta^*_\lambda g \in V_{n-1, \lambda+1}$. So 
$\Theta^*_\lambda(h) \in V_{n-1, \lambda+1} \cap V^\perp_{n-1, \lambda+1} = \{0\}$. Thus $h\in \ker \Theta^*_\lambda = V_{1, \lambda-1} \subseteq V_{n, \lambda-1}$. Hence $h\in V_{n,\lambda-1} \cap V^\perp_{n, \lambda-1} = \{0\}$. Thus $h=0$, and hence $f=g\in V_{n,\lambda-1}.$ This proves that $\Theta^*_\lambda f \in V_{n-1, \lambda+1} \implies f \in V_{n,\lambda-1}$. Conversely, by Lemma \ref{Decomposition of Theta}, $f\in V_{n, \lambda-1} \implies \Theta^*_\lambda f \in V_{n-1, \lambda+1}$.   
So, ${\Theta^*_\lambda}^{-1}(V_{n-1, \lambda+1} ) = V_{n, \lambda-1}.$ 
\end{proof}

\begin{proof}[Proof of Theorem 5.2] The formula \eqref{prod1} implies that that two definitions of $\theta_{\lambda,n}$ are equivalent. Since, by \cite{scal}, $\theta_{\lambda}$ is an inner function for each $\lambda > 1$, the second formula for $\theta_{\lambda, n}$ shows that $\theta_{\lambda, n}$ is a pointwise product of finitely many inner functions. Therefore, $\theta_{\lambda,n}$ is an inner function.  Hence the description of the Nagy-Foias model  \cite{Na-Fo} for the cnu contractive operator $T$ with characteristic function $\theta_{\lambda,n}$ simplifies as follows. 

Let $T$ be the compression of $I \otimes S:\mathcal H^{(\lambda-1)}\otimes H^2 \to \mathcal H^{(\lambda-1)}\otimes H^2$ (where, as before, $S$ is the multiplication operator on $H^2$) to the subspace $\ker \Theta^*_{\lambda,n} = V_{n,\lambda-1}$ (Lemma \ref{lemma 5.5}). Then the characteristic function of $T$ is $\theta_{\lambda, n}$. 

Let us identify $\triangle$ with $\mathbb D$ via the map $z \mapsto (z,z)$, $z\in \mathbb D$. Define the map $J:V_{n, \lambda-1} \to {\rm Hol}(\mathbb D, \mathbb C^n)$ by 
$$
Jf = \Big ( \frac{1}{(\lambda-1)_i} \frac{\partial^if}{\partial z^i}\Big |_{ \triangle} \Big)_{0\leq i < n}, f\in V_{n,\lambda-1}.
$$
Let $\mathcal H$ be the image of $J$. It is immediate from the definition of $V_{n,\lambda-1}$ that $J$ is a bijection between $V_{n,\lambda-1}$ and $\mathcal H$. Use this bijection to transfer the inner product from $V_{n,\lambda-1}$ to $\mathcal H$.  This converts $\mathcal H$ into a Hilbert space, and $J: V_{n,\lambda-1} \to \mathcal H$  is a unitary. Following the argument in \cite{DMV}, it is easy to see that (a) $\mathcal H$ is a functional Hilbert space with reproducing kernel $K:\mathbb D \times \mathbb D \to \mathbb C^{n\times n}$ given by 
$$K(z,w) = (1-z\bar{w})^{-1} B^{(\lambda-1, \underline{e}_0)}(z,w),\,\, \underline{e}_0 = (1,0,\ldots,0) \in \mathbb R^n,$$ 
and (b) $J$ intertwines $T$ with the multiplication operator $M$ on $\mathcal H$. This is a minor variation of the jet construction discussed in \cite{DMV}. 

Now, if $\underline{\mu}$ is the special parameter described in Notation 5.1, then, in the notation of \eqref{mudoubleprime}, we have $\underline{\mu}^{\prime\!\prime} = \underline{e}_0$.  Therefore, by Equation \eqref{God}, we have $(1-z\bar{w}) B^{(\lambda, \underline{\mu})}(z,w)= B^{(\lambda-1, \underline{e}_0)}(z,w)$. Hence  $K=B^{(\lambda, \underline{\mu})}$, and therefore $\mathcal H= \mathcal H^{(\lambda, \underline{\mu})}$, $M= M_{\lambda, n}$. Thus $M_{\lambda,n}$ is unitarily equivalent to $T$ via $J$. Therefore, $M_{\lambda,n}$ is a cnu contraction and the characteristic function of $M_{\lambda,n}$ coincides with the characteristic function $\theta_{\lambda,n}$ of $T$.  
\end{proof}
Theorem 5.2 prompts us to pose: 
\begin{conjecture}
The characteristic function of any multiplicity free cnu contractive associator in $B_n^*(\mathbb D)$ is the pointwise product of the characteristic functions of finitely many generic multiplicity free contractive associators from ${\cup}_{1\leq m \leq n}B_m^*(\D).$ 
\end{conjecture}

\subsubsection*{Acknowledgement} This work is the result of the joint research initiated at the Indian Statistical Institute between the first and the third authors. It was further investigated in the PhD thesis \cite{SH} of the second author and was completed while the first author was visiting the Indian Institute of Science. We thank both the Indian Statistical Institute and the Indian Institute of Science for providing a stimulating environment to complete this work.

\end{document}